\documentclass{article}
\usepackage{a4wide}
\usepackage{amssymb,amsmath,amsfonts,amsthm,mathrsfs}
\usepackage{color,graphicx,subfigure}
\usepackage{booktabs}
\usepackage{tikz}
\usepackage{tikz-3dplot} 
\usepackage{pgfplots}
\pgfplotsset{compat=1.13}
\graphicspath{{./}{figures/}}

\newtheorem{theorem}{Theorem}
\newtheorem{lemma}{Lemma}[section]
\newtheorem{corollary}[lemma]{Corollary}
\newtheorem{proposition}[lemma]{Proposition}
\newtheorem{remark}[lemma]{Remark}

\def\Cl{\mathfrak{C}_l}
\def\Cone{\mathfrak{C}_1}
\def\e{\varepsilon}
\def\wto{\rightharpoonup}
\def\wtos{\stackrel{*}{\wto}}
\def\vp{\varphi}
\def\R{\mathbb{R}}

\newcommand{\Sp}{\mathbb{S}}
\def\Z{\mathbb{Z}}

\def\H{\mathcal{H}}

\def\E{\mathcal{E}}

\def\LM#1{\hbox{\vrule width.2pt \vbox to#1pt{\vfill \hrule width#1pt
height.2pt}}}
\def\LL{{\mathchoice {\>\LM7\>}{\>\LM7\>}{\,\LM5\,}{\,\LM{3.35}\,}}}
\def\restr{{\LL}}
\def\Div{\textup{div}\,}
\def\dom{\textup{dom}\,}
\def\dist{\textup{dist}\,}

\def\spt{\textup{spt}}
\def\utheta{\underline{\theta}}
\def\dx{\delta_x}
\def\dt{\delta_\theta}
\def\M{\mathcal{M}} %% measures
\def\Om{\Omega}
\newcommand\OS[1][]{{\Om_{#1}\times\Sp^1}}

\def\ov{\overline}
\def\ale{\textit{a.e.}}
\newcommand{\sca}[2]{{\left\langle {#1},{#2}\right\rangle}}

\def\12{\tfrac{1}{2}}
\def\ha{\frac{1}{2}}

\newcommand{\norm}[2][]{\|{#2}\|_{{#1}}}

\def\i{\mathbf{i}}
\def\j{\mathbf{j}}
\def\I{\mathcal{I}}
\def\J{\mathcal{J}}
\def\F{\mathscr{F}}

\begin{document}

\title{Total Roto-Translational Variation} \author{Antonin
  Chambolle\thanks{CMAP, Ecole Polytechnique, CNRS, 91128 Palaiseau,
    France. \newline email: \texttt{antonin.chambolle@cmap.polytechnique.fr}}, and Thomas
  Pock\thanks{Institute for Computer Graphics and Vision, Graz
    University of Technology, 8010 Graz, Austria and Center for
    Vision, Automation \& Control, AIT Austrian Institute of
    Technology GmbH, 1220 Vienna, Austria. email:
    \texttt{pock@icg.tugraz.at} }}

 \maketitle

\begin{abstract}
  We consider curvature depending variational
  models for image regularization, such as Euler's elastica.
  These models are known to provide
  strong priors for the continuity of edges and hence have
  important applications in shape- and image processing. We consider a
  lifted convex representation of these models in the roto-translation
  space: %$\OS$.
  In this space, curvature depending variational energies are represented
  by means of a convex functional defined on divergence free vector
  fields. The line energies are then easily extended to any scalar function.
  It yields
  % Therefore, our model can be seen as 
  a natural
  generalization of the total variation to the roto-translation
  space. As our main result, we show that the proposed convex
  representation is tight for characteristic functions of smooth
  shapes. We also discuss cases where this representation
  fails. For numerical solution, we propose a staggered grid
  discretization based on an averaged Raviart-Thomas finite elements
  approximation. This discretization is consistent, up to minor
  details, with
  the underlying continuous model. The resulting non-smooth convex
  optimization problem is solved using a first-order primal-dual
  algorithm. We illustrate the results of our numerical algorithm on
  various problems from shape- and image processing.
\end{abstract}

\section{Introduction}
It was observed at least since~\cite{NitzbergMumfordShiota} that
line energies such as Euler's ``Elastica'' (here for a smooth
curve $\gamma$):
\[
\int_{\gamma} [1+]\; \kappa^2 ds
\]
could be natural regularizers for the completion of
missing contours in images. This idea was based in particular
on observations of G.~Kanizsa~\cite{kanizsa1980grammatica,Kanizsa-en}
about the way our perception can ``invent'' apparent contours.
The model in~\cite{NitzbergMumfordShiota}
(see also~\cite{EseShen2002,EseMarch2003}
for interesting attempts to solve it with phase-field methods)
was variational in nature
and since then, many attempts have been proposed to study
and address the
minimization of such line energies, both theoretically and numerically.

From the theoretical point of view, the lower-semicontinuity of
these energies already is a challenge, which has been studied
in many papers (in connexion to the applications to computer vision)
since at least the 90s~\cite{Bellettini-MS-1993,Bel-Mug2004,
BelleMarch2004,Bel-Mug2005,BelleMarch2007,Bel-Mug2007,DayrensMasnouNovaga17}.
It is shown already in~\cite{Bellettini-MS-1993} that
 the boundary of many sets with cusps
can be approximated by sets with smooth boundaries of bounded energy,
showing that the relaxation of the Elastica for boundaries of
sets is already far from trivial. The study of this lower semi-continuity
of course enters the long history of the study
of general curvature dependent energies
of manifolds and in particular the Willmore energy~\cite{Willmore}.

Quite early, it has been suggested to lift the manifold in a larger
space where a variable represents its direction or orientation, by means
in particular (for co-dimension one manifolds) of the Gauss map
$(x,\nu(x))$, $\nu(x)$ being the normal to the manifold at $x$~\cite{Anzellotti1990,AST1990}. Such approach allows to study very general curvature energies
and has been successfully used for establishing lower-semicontinuity
and existence
results~\cite{AST1990,AnzeDelladio1995,Delladio1997a,Delladio1997b},
and in particular to lines energies such as ours in higher codimension~\cite{AcerbiMucci2017a,AcerbiMucci2017b} or in dimension 2~\cite{Delladio1997b}.
We must mention also in this class %of approaches 
an older approach based
on ``curvature varifolds'' (which is very natural since varifolds
are defined on the cross product of spatial and directional variables)
which has allowed to show existence results since
the 80s~\cite{Hutchinson1986}, see also~\cite{Mantegazza}.
%% cite later on Ambrosio Masnou

Interestingly, it was understood much earlier~\cite{HubelWiesel1959}
that (cats') vision was functioning in a similar way
(\textit{eg.}, using a sort of ``Gauss map''),
thanks to neurons sensitive to particular directions which were found
to be stacked inside the visual cortex into ordered columns,
making us sensitive to changes of
orientations (and thus curvature intensity). These findings
(which might explain some of Kanizsa's experiments) inspired
some mathematical models quite early~\cite{KD1987}, however
they were formalized into a consistent geometric interpretation
later on~\cite{PetitotTondut,SartiCittiPetitot,SartiCitti,Cittietal,MoiseevSachkov,Duitsetal14}.
The idea of these authors is to lift 2D curves in
the ``Roto-translation'' space (which is the group of
rotations and translations, however what is more important %interesting
here is less its group structure than the associated
\textit{sub-Riemanian} metrics)
which can be identified, for our purposes, with $\OS$ where
$\Om\subset \R^2$ is the spatial domain where the image is
defined and $\Sp^1$ parameterizes the local orientation.

The natural metric in that space prevents from moving spatially
in a direction other than the local orientation, which
makes it singular (hence ``sub''-Riemanian). Diffusion
and mean curvature flow in this particular metrics were successfully
used as efficient methods for image ``inpainting'', which is
the task of filling in a gap in an
image~\cite{Franken2009,Boscainetal,Prandietal,SharmaDuits2015}.
Indeed, such diffusion
naturally extends missing level sets into smooth curves, and
even allows for crossing, since curves with different directions ``live''
in different locations in the Roto-translation space.
This seems to improve in difficult situations (\textit{eg.}, crossings)
upon more classical diffusion models for
inpainting~\cite{ChanShen2001,bertalmio2000image}
(or Weickert's ``EED''~\cite{Weickert1996,WeickertCat2012,Schmaltzetal2014}
which can propagate
directional information quite smoothly accross inpainted regions).

In computer science, similar ideas were successfully 
implemented in discrete graphs (with nodes representing a
spatial point and orientation)
in order to minimize curvature-dependent contour energies
\cite{schoenemann2007introducing,el2010fast,el2010optimization,SMC2011,Krueger2011,SKMC2012,SMC2012,ElZgrady2016}. Our current work is closer to these
approaches, although set up in the continuous setting, as we
want to represent and solve variational models involving curvature terms,
were the unknown are scalar (grey-level) functions.
A continuous approach similar to these discrete ones proposes
to compute minimal paths in the Roto-translational metric
by solving the corresponding eikonal equations~\cite{Mirebeau2014,mirebeau_FM_2017,DaMirebeauCohen2017}
(the goal being more here to find paths on an image than to complete
boundaries). This is strongly connected to geometric control
problems (such as parking a car), and these connections have led
to the study many interesting metrics in these settings,
with impressive applications
to imaging problems such as the extraction of networks of vessels and fibers in 2D and 3D medical images via fast, globally optimal sub-Riemannian and sub-Finslerian geodesic tracking in the roto-translation group~\cite{BekkersDuitsMS2015,Duits-et-al2016}.
% the extraction of networks of vessels in 3D medical images
\smallskip

As said, we wish to extend these variational methods to level
sets representations.
The family of problems we are interested in is introduced
in a paper by Masnou and Morel~\cite{masnou1998level}
(see also~\cite{MasnouMorel06}), which addresses the
problem of image inpainting. Their initial idea is to minimize,
in the inpainting domain $D$, an energy of the form
\[
\int_D \left(1+\left|\Div \frac{Du}{|Du|}\right|^p\right)|Du|
\]
for some $p\ge 1$, where $u$ is a bounded variation function and
appropriate boundary conditions are given. A theoretical
study of this energy, which in general is not lower-semicontinuous (lsc),
is found in~\cite{AmbrosioMasnou}, in particular it is shown
that if $p>1$ (in dimension $2$), it is lsc on $C^2$ functions.
Interestingly, also this study relies on a space/direction representation
(and more precisely on varifolds). An interesting co-area
formula for the relaxed envelope is also shown in~\cite{MasnouNardi2013}.

There have been many attempts to numerically solve Masnou and Morel's
model, which is very difficult to tackle, being highly non convex.
Most of the authors introduce auxiliary
variables~\cite{Ballesteretal2001,ChanKangShen2002,Ballesteretal2003},
for instance representing the orientation, which is already close
to the idea of the Gauss map or Roto-translational representation.
Recent techniques
based on Augmented Lagrangian methods (for coupling the auxiliary
variables)~\cite{TaiHahnChung,ZTC2013,Yashtini2015,YashKang16,GlowinskiPanTai2016,Glowinski2016,BTZ2017}
have shown to be quite efficient, despite the lack
of convexity and hence convergence guaranties.
% "error estimates" \cite{ChanKang2006}
\smallskip

What we propose here is to rather introduce a functional which  may be
sees as a convex relaxation of Masnou and Morel's. 
The method we propose is based on two ingredients: A lifting
in the Roto-translation space of curves which allows to write the
Elastica energy (or any convex function of the curvature) as
a convex function, as in classical Gauss map based approaches, 
and (formally) a decomposition of functions as sum of characteristic functions
of sets with finite energy. This allows to define a convex functional
which is defined on grey-level functions and penalises the curvature
of the level lines. It is however easy to check that it is
in general strictly below Masnou and Morel's functional, in
particular a function of finite energy needs not have its level
sets regular (in practice, they could be intersections of regular
sets, in addition to having possible cusps, for the same
reasons as in~\cite{Bellettini-MS-1993}), see for example
Figures~\ref{fig:infinity} and~\ref{fig:inpainting-einstein-pablo}(e).
The only theoretical
result which we can show is that $C^2$ curves are tightly represented
by our convex relaxation.

There is a close relationship between our approach and the
functionals in~\cite{BrPoWi13} (``TVX'') and~\cite{BrPoWi15}, based
on similar representations (but~\cite{BrPoWi13} omits to preserve the
boundary of the lifted current). In fact, the functional we 
build is, as we  show
further on, a new expression of the previous convex relaxation
of the Elastica energy (and variants) proposed in~\cite{BrPoWi15}.
This stems from the identity of the corresponding dual problems.
Although it boils down to the same
energy, it is introduced in a much simpler way, as the primal
expression in~\cite{BrPoWi15} requires to work in a space
where the point, the tangent and the curvature are lifted
as independent variables. 
As a consequence, we also can provide
a simpler discretization, and the tightness of the relaxation
for $C^2$ sets was unnoticed in~\cite{BrPoWi15}.
Eventually, we should mention that part of what we propose here could
be generalized to arbitrary dimension or co-dimension, following
the techniques in~\cite{AST1990}. However, computationally, our 
construction for 1D curves in the plane
already needs to work in a 3D space, and more complex models
seem at this moment intractable. Our future research will
rather focus on improving the discretization and the optimization
of the bidimensional case.
\smallskip

The paper is organised as follows: In Section~\ref{sec:func} we describe
the lifting and introduce our functional. We state our main
result, which shows that characteristic functions of $C^2$ sets are well
represented by our convex relaxation (Theorem~\ref{th:C2}).
In the next Section~\ref{sec:prop}
we show an approximation result with smooth functions, and
compute a dual representation of the functional (which
shows it is identical to~\cite{BrPoWi15}).
Then, in Section~\ref{sec:numexp}, we describe how the functional
is discretized and show a few experimental results.
The last Section~\ref{sec:proof} is devoted to the proof if Theorem~\ref{th:C2}.
Some technical tools are found in the Appendix, in particular,
Appendix~\ref{app:consist} shows, up to (hopefully minor) transformations,
the consistency of our implementation in Section~\ref{sec:numexp}.

\section{The functional}\label{sec:func}

We consider $\Om$ a bounded open domain in the plane, and
$E\subset\Om$ a set with $C^2$ (and to simplify, connected) boundary.
Assume that this boundary $\partial E$ is given by a parameterized
curve $\left(x_1(t), x_2(t)\right)$ with parameter $t \in [0,1]$. The
(extrinsic) curvature $\kappa_E$ of $\partial E$ is classically
defined as the ratio between the variation of the tangential angle
$\theta$ and the variation of its arc length $s$, that is
\[
\kappa_E = \frac{d\theta}{ds} =
\frac{\frac{d\theta}{dt}}{\frac{ds}{dt}}
\frac{\frac{d\theta}{dt}}{\sqrt{\left(\frac{dx_1}{dt}\right)^2 +
    \left(\frac{dx_2}{dt}\right)^2}}\;.
\]
The main idea of the lifting is now to consider a higher dimensional
representation of the parametric curve in the 3D roto-translation (RT)
space, which is obtained by adding the tangential angle $\theta$ as an
additional dimension to $\Om$. In this space, we now consider a
parameterized 3D curve $\left(x_1(t), x_2(t), \theta(t)\right)$ which
lifts the boundary $\partial E$ to the roto-translational
space. Figure~\ref{fig:rtspace} shows an example where we lift the
boundary of a disk (which is a circle) to the RT space. Observe that
the lifted boundary is represented by a 3D helix.

Now, we define for all $t \in [0,1]$ the tangential vector $p(t) =
(p^x(t),p^\theta(t))$ with
\[
p^x(t) = \left(\frac{dx_1(t)}{dt}, \frac{dx_2(t)}{dt}\right), \quad
p^\theta(t) = \frac{d\theta(t)}{dt}, \quad |p^x(t)| = \sqrt{\left(\frac{dx_1(t)}{dt}\right)^2 +
  \left(\frac{dx_2(t)}{dt}\right)^2}.
\]
The curvature is therefore given by
\[
\kappa_E(t) = \frac{p^\theta(t)}{|p^x(t)|},
\]

In this work we consider $f:\R\to [0,+\infty]$ a convex,
lsc function and want to define a convex lsc extension (to grey-level
valued functions) of energies of the type
\[
E\mapsto  \int_{\partial E} f(\kappa_E) d\H^1
\]
where $E\subset\Om$ is a set with $C^2$ boundary, and $\kappa_E$ is
the curvature of the set. Using our tangential vector $p(t)$, it is
easy to see that the energy can be (formally) written as
\[
\int_{\partial E} f(\kappa_E) d\H^1 = \int_0^1
f(p^\theta/|p^x|)|p^x|dt = 
\int_{\OS} f(\tau^\theta/|\tau^x|) |\tau^x| d\H^1\restr \Gamma_E
\]
where $\Gamma_E=(x,\theta)([0,1])$ is the lifted curve and
$\tau(x,\theta)$ its normalized tangential vector, given
by $\tau(x(t),\theta(t))=p(t)/|p(t)|$ for all $t\in[0,1]$.
A precise definition of the energy will be given below.

\begin{figure}[ht!]
  \centering
  \begin{tikzpicture}[scale = 1.2]
  \begin{axis}
    [ view={70}{50},
      enlargelimits = false,
      xlabel={$x_1$},
      ylabel={$x_2$},
      zlabel={$\theta$},
      xmin=-1.4,xmax=1.4,
      ymin=-1.4,ymax=1.4
    ]

    % draw filled circle
    \addplot3[black,line width=2.0pt, fill=lightgray,variable=t,domain=0:1,samples=30]
    ({cos(deg(2*pi*t))},{sin(deg(2*pi*t))},0);
    
    % draw circle in RT space
    \addplot3[gray,line width=2.0pt,variable=t,domain=0:1,samples=30, samples y=0]
    ({cos(deg(2*pi*t))},{sin(deg(2*pi*t))},{2*pi*t});

    % draw support lines
    \pgfplotsinvokeforeach{0,0.1,...,1} {
      \draw[gray] ({cos(deg(2*pi*#1))},{sin(deg(2*pi*#1))},{2*pi*#1})
      -- ({cos(deg(2*pi*#1))},{sin(deg(2*pi*#1))},0);

      \node at (axis cs:0,0,0) {$E$};
      \node at (axis cs:-0.9,-0.8,0) {$\partial E$};

      \draw[gray] (1,0,0) -- (1,0.5878,0);
      \node[gray] at (axis cs:1.2,0.3,0) {$dx_2$};
      
      \draw[gray] (1,0.5878,0) -- (0.8090,0.5878,0);
      \node[gray] at (axis cs:0.9,0.8,0) {$dx_1$};

      \draw[gray] (0.8090,0.5878,0) -- (0.8090,0.5878,0.6);
      \node[gray] at (axis cs:0.6,0.45,0.2) {$d\theta$};
      
    }
  \end{axis}
\end{tikzpicture}
  \caption{The gray line is the lifting of the boundary $\partial E$
    of the disk $E$ to the roto-translational space
    $\OS$.}\label{fig:rtspace}
\end{figure}
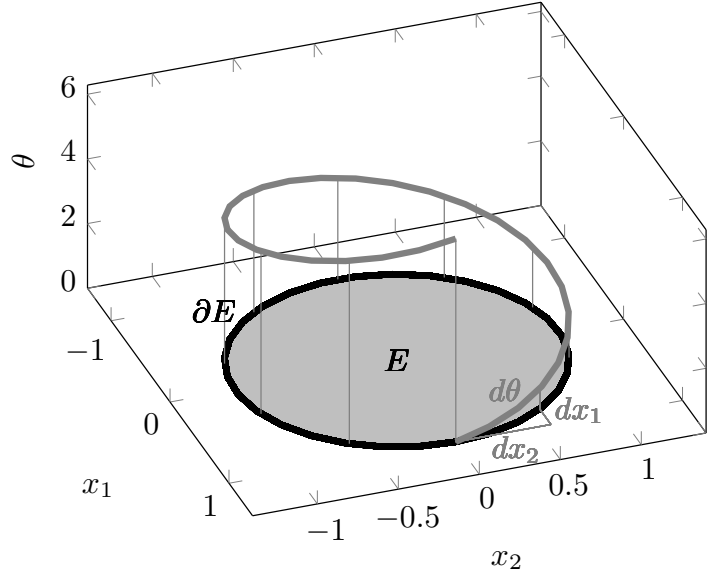

Let us briefly discuss three instances of energies that will typically
appear in applications. In all cases $\alpha > 0$ will be a tuning
parameter that can be used to balance the influence of the curvature
term with respect to the length.
\begin{enumerate}
  \item $f_1(t) = 1+\alpha |t|$. This energy penalizes the arclength plus
    the absolute curvature, hence we might expect that this type of
    energy will allow also for corners. This type of energy has been
    studied with a different approach in~\cite{BrPoWi13}.
    The energy is given by:
    \[
    \int_{\OS} f_1(\tau^\theta/|\tau^x|)|\tau^x|d\H^1\restr\Gamma_E = 
    \int_{\OS} |\tau^x| + \alpha |\tau^\theta| d\H^1\restr\Gamma_E.
    \]
  \item $f_2(t) = \sqrt{1+\alpha^2|t|^2}$. This energy penalizes the
    arclength of the lifted curve in the RT space which, in some
    sense, corresponds to the ``Total Roto-translational
    Variation''. It yields the energy:
    \[
    \int_{\OS} f_2(\tau^\theta/|\tau^x|)|\tau^x|d\H^1\restr\Gamma_E
    = \int_{\OS} \sqrt{|\tau^x|^2 + \alpha^2 |\tau^\theta|^2} d\H^1\restr\Gamma_E
    \]
    which is nothing but the length of the lifted curve $\Gamma_E$ in
    a Riemanian metric. This functional was considered for
    perceptional completion problems probably first in~\cite{SartiCitti}.
%% comments of Remco 
    It coincides with
    sub-Riemannian~\cite{SartiCitti,Duitsetal14,Boscainetal14,BekkersDuitsMS2015}
    and more precisely the sub-Finslerian models in~\cite{Duits-et-al2016} (with the constraint of positive direction on the velocity).

  \item $f_3(t) = 1+\alpha^2|t|^2$. This is the classical Eulers's
    Elastica energy studied for example
    in~\cite{NitzbergMumfordShiota,Bellettini-MS-1993,MasnouMorel06,schoenemann2007introducing,BrPoWi15}
    and many other works already mentioned in our introduction.
    \[
    \int_{\OS} f_3(\tau^\theta/|\tau^x|)|\tau^x| d\H^1\restr\Gamma_E
    = \int_{\OS} |\tau^x| + \alpha^2 \frac{|\tau^\theta|^2}{|\tau^x|} d\H^1\restr\Gamma_E.
    \]
    Observe that the second term in the energy is a ``quadratic over
    linear'' function which is still convex, provided $\tau^x$
    is constrained in a half space or a half line (which of course
    will depend on the $\theta$ variable).
\end{enumerate}
In all these examples, we find that the energy which we are interested
in ends up represented possibly as a \textit{convex} function of the
measure $\sigma = \tau\H^1\restr\Gamma_E$. Moreover, this $\sigma$ is
not an arbitrary measure. In particular, it satisfies two important
constraints: first, by construction, it is a circulation and has zero
divergence in $\OS$ (it can have source terms on $\partial\OS$):
indeed, given any smooth function $\psi$, $\int\nabla\psi\cdot\sigma=
\int_{\Gamma_E}\partial_\tau\psi d\H^1$ vanishes if $\Gamma_E$ is a
closed curve, or if it ends on the boundary and $\psi$ has compact
support.  Second, its marginals in $\OS$, which we can formally denote
$\int_{S^1}\sigma$ (and which are also divergence free), coincide with
a $90^\circ$-rotation of the measure $D\chi_E$. We will show now how
to generalize this construction to arbitrary sets or functions (with
bounded variation).

In the whole paper, we assume that there exists $\gamma>0$ such that
\begin{equation}\label{eq:controlbelow}
f(t)\ge\gamma\sqrt{1+t^2},\quad \forall\ t\in\R.
\end{equation}
We also introduce the \textit{recession function}
\[
f^\infty(t) = \lim_{s\to +\infty} \frac{1}{s}f(st) 
\]
which is a convex, one-homogeneous function possibly infinite
on $(-\infty,0)$ or/and $(0,+\infty)$. It is easy to check that
it is the support function of $\dom f^*$, the domain of the convex
conjugate $f^*$ of $f$: 
\begin{equation}\label{eq:recession}
f^\infty(t) =  \sup_{y\in\dom f^*}ty.
\end{equation}
Indeed, for $y\in\dom f^*$, $s>0$,
\[
ty = \frac{1}{s} (st)y \le \frac{1}{s}f(st)  + \frac{1}{s}f^*(y)
\stackrel{s\to\infty}{\longrightarrow} f^\infty(t)
\]
as $s\to\infty$; on the other hand, as $f^*\ge -\gamma$
thanks to~\eqref{eq:controlbelow}, for $s>0$ one has
\[
\frac{1}{s}f(st) = \sup_{y\in\dom f^*} ty - \frac{1}{s}f^*(y)
\le \sup_{y\in\dom f^*} ty + \frac{\gamma}{s} 
\]
showing that $f^\infty(t)\le \sup_{y\in\dom f^*}ty$.

We let then, for $p=(p^x,p^\theta)\in \R^3$ with $p^x\neq 0$,
\begin{equation}\label{eq:defh}
h(\theta,p) =
\begin{cases} |p^x|f(p^\theta/|p^x|) & \textup{ if } p^x\in \R_+\utheta, p^x\neq 0,\\
f^\infty(p^\theta) & \textup{ if }p^x=0,\\
+\infty & \textup{ else.}
\end{cases}
\end{equation}
Here, $\utheta$ denotes the unit planar vector $( \cos\theta,\sin\theta)^T$ 
(by a slight abuse of notation, we will  also denote
in this way the vector $(\cos\theta,\sin\theta,0)^T\in\R^3$).
It is then classical~\cite[\S~13]{Rockafellar}
 % dal maso, or giaquinta modica soucek?
and easily follows from~\eqref{eq:recession} that one has
\begin{equation}\label{eq:dualh}
h(\theta,p) = \sup \left\{ \xi\cdot p :  \xi^x\cdot\utheta \le -f^*(\xi^\theta)
\right\},
\end{equation}
that is, the one-homogeneous function $h$ is the support function
of the convex set in the right-hand side of~\eqref{eq:dualh}.
Indeed, first of all, the sup in~\eqref{eq:dualh}
is $\infty$ if $p^x$ is not $\lambda\utheta$,
$\lambda\ge 0$. While if it is of this form, then (taking the
suprema over the $\xi$'s which satisfy $\xi^x\cdot\utheta \le -f^*(\xi^\theta)$)
\[
\sup_{\xi\,:\,\xi^x\cdot\utheta \le -f^*(\xi^\theta)} \xi^x\cdot\utheta\lambda + \xi^\theta p^\theta =
\sup_{\xi^\theta} \xi^\theta p^\theta -\lambda f^*(\xi^\theta)
= \begin{cases}
\lambda f(p^\theta/\lambda) & \textup{ if } \lambda>0\\
f^\infty(p^\theta) & \textup{ else.}
\end{cases}
\]
Observe that~\eqref{eq:controlbelow} yields that for all $p\in\R^3$,
\begin{equation}\label{eq:controlbelowh}
h(\theta,p) \ge \gamma |p|
\end{equation}

We now introduce the functional
\begin{equation}\label{eq:deFu}
F(u) = \inf\left\{ \int_{\OS} h(\theta,\sigma)dx d\theta\,:\,
\Div \sigma=0, \int_{\Sp^1}\sigma^x d\theta = Du^\perp \right\}.
\end{equation}
Here $x^\perp=(x_2,-x_1)$ is a $90^\circ$ rotation in the plane.
The last condition is understood as follows: for all
$\vp\in C_c^1(\Om;\R^2)\subset C_c^1(\OS;\R^2)$,
one has
\begin{equation}\label{eq:constraintDu}
\int_{\OS} \vp^\perp\cdot\sigma^x  = \int_\Om \vp\cdot Du
= -\int_\Om u\Div \vp dx.
\end{equation}
In particular in general the fields $\sigma$ appearing in~\eqref{eq:deFu}
are free-divergence bounded Radon measures in $\OS$ with values in $\R^3$ 
(we denote $\M^1(\OS;\R^3)$ the space of such measures)
and the proper way to write the integral is rather
\[
\int_{\OS} h(\theta,\sigma) = 
\int_{\OS} h\left(\theta,\frac{\sigma}{|\sigma|}\right)d|\sigma|
\]
where $\sigma/|\sigma|$ is the Radon-Besicovitch derivative of $\sigma$
w.r.~its total variation~$|\sigma|$.
Notice that  for any $u$ and any admissible $\sigma$,
thanks to~\eqref{eq:controlbelowh},
\[
F(u) \ge 
\int_{\OS} h(\theta,\sigma)
 \ge \gamma\int_{\OS} |\sigma|
\ge % \gamma \int_{\Om} |Du^\perp| =
\gamma \int_{\Om} |Du| 
\]
because of~\eqref{eq:constraintDu}, taking the supremum
with respect to the test function $\vp$ with $|\vp(x)|\le 1$ everywhere.
This shows that $F(u)$ bounds the $BV$ seminorm.

We denote by $K(\theta)$ the closed convex set
of $\R^3$ whose support function is $h(\theta,\cdot)$, as already
seen it is given by
\[
K(\theta)= \left\{ \xi\in\R^3 :  \xi^x\cdot\utheta \le -f^*(\xi^\theta)
\right\}.
\]
Remark that \eqref{eq:controlbelow}
implies that $f^*(s)\le -\sqrt{\gamma^2-s^2}$
for $|s|\le\gamma$, implying in particular that $0$
is in the interior of $K(\theta)$.

We then let (following for instance~\cite{BouchitteValadier,Rockafellar71})
\[
K = \left\{\vp \in C_c^0(\OS;\R^3): \vp(x,\theta)\in K(\theta) \;\forall
(x,\theta)\in\OS\right\}.
\]
Then, we claim that for any measure $\sigma\in \M^1(\OS;\R^3)$
\begin{equation}\label{eq:intdualh}
\int_{\OS} h(\theta,\sigma)
= \sup_{\vp\in K} \int_{\OS} \vp\cdot\sigma,
\end{equation}
showing in particular that this integral is a lower-semicontinuous
function of $\sigma$ (for the weak-$*$ convergence).
To prove this claim (which is standard), first
observe that if $\vp\in K$, then for all $\psi^-\in C_c^0(\OS;\R_+)$
and $\psi\in C_c^0(\OS)$, also $\vp-\psi^-\utheta+\psi\utheta^\perp\in K$.
One deduces easily that the supremum in~\eqref{eq:intdualh} is infinite
if one of the measures $(\utheta\cdot\sigma^x)^-$ or $\utheta^\perp\cdot\sigma^x$
does not vanish. If both vanish, it means that $\sigma=(\lambda\theta,\sigma^\theta)$ for some nonnegative measure $\lambda$. Hence the supremum becomes
%\begin{multline*}
\[
\sup_{\vp\in K}\int_{\OS} (\vp^x\cdot\theta) \lambda + \vp^\theta \sigma^\theta
%\\
= \sup_{\psi+f^*(\vp^\theta)\le 0}\int_{\OS} \psi \lambda+
 \vp^\theta \sigma^\theta
= \sup_{\vp^\theta} \int_{\OS} \vp^\theta \sigma^\theta - f^*(\vp^\theta) \lambda,
\]
%\end{multline*}
and~\eqref{eq:intdualh} is then deduced in a standard way (with
a Besicovitch covering of $\OS$ with respect to the measure
$|\sigma|$ by balls where $\lambda,\sigma^\theta$ are ``almost constant'', and
choosing then for
$\vp$ the ``right function'' in each ball of the covering---the
fact that zero is in the interior of $K$, which allows to multiply
functions in $K$ by a cut-off, is important here).

We can deduce that also $F$ is lower semicontinuous:
consider a sequence $(u_n)$ with $F(u_n)\le c<\infty$.
Then $u_n$ is bounded in $BV$ and converges
(up to a constant and a subsequence) to some $u$ in $L^1(\Om)$.

If $\sigma_n$ reaches the value of $F(u_n)$ up to $1/n$, one has
that
\[
\int_{\OS}  |\sigma_n| \le c <\infty
\]
so that $\sigma_n\stackrel{*}{\rightharpoonup} \sigma$ (up to a subsequence), as measures,
and obviously $\Div \sigma=0$. Clearly, also~\eqref{eq:constraintDu}
passes to the limit. Hence, by lower semicontinuity,
\[
F(u)\le \int_{\OS} h(\theta,\sigma) 
\le\liminf_n \int_{\OS} h(\theta,\sigma_n)  = \liminf_n F(u_n).
\]
This shows that $F$ defines a convex, lower semicontinuous functional
on $BV(\Om)$. Our main theoretical
result is the following theorem, which shows
that if the argument $u$ is the characteristic function of a smooth enough set,
then $F(u)$ coincides, as expected,
with a curvature-dependent energy of the boundary of the set $\{u=1\}$.
The proof of this result is postponed to Section~\ref{sec:proof}.
\begin{theorem}\label{th:C2}
Let $E\subset\Omega$ be a set with $C^2$ boundary. Then
\begin{equation}\label{eq:goodrelaxE}
F(\chi_E) = \int_{\partial E\cap \Omega}  f(\kappa_E(x))d\H^1(x).
\end{equation}
\end{theorem}

\begin{remark} \textup{One could hope that $F$ coincides with the lower semicontinuous
envelope of its restriction to $C^2$ sets, with respect to the $L^1$ convergence.
However, simple examples show that it is not the case.
Many examples where it fails are found in~\cite{BrPoWi15}.
In general, we expect the relaxation $F$ to be strictly below the
relaxations found in the literature since~\cite{Bellettini-MS-1993},
based on the ($L^1$) approximation of sets with smooths sets.
In particular, in the case of~\cite[Fig.~4.2]{Bellettini-MS-1993},
our relaxation will certainly be below twice the energy reached
by the pattern in~\cite[Fig.~1.2]{Bellettini-MS-1993}, while the
energy reached in~\cite[Fig.~4.2]{Bellettini-MS-1993} is strictly larger.
%\\\textit{Here: Discussion on the fact that the relaxation $F$ is in general
%lower than the lsc envelope of the rhs, cite Bellettini-Mugnai's papers.}
}\end{remark}

\begin{remark}\textup{If $f$ has growth one and is even,
one expects that the same
result holds for piecewise $C^2$ sets.}
\end{remark}

\begin{remark}\label{rem:uC2} \textup{We believe
that the proof below could be extended to show that
if $u\in BV(\Omega)$ is a function with $C^2$ level sets,
and such that there exists a continuous function $\kappa(x)$ which coincides
with the curvature of $\partial\{ u>s\}$ for all level $s\in\R$,
then one should have
\begin{equation}\label{eq:goodrelaxu}
F(u) =\int_{\R} \int_{\partial\{u>s\}} f(\kappa_{\partial\{u>s\}}(x))d\H^1(x)
= \int_\Om f(\kappa(x))d|Du|(x).
\end{equation}
Hence in that case our functional would coincide with Masnou and Morel's~\cite{MasnouMorel06}.
A typical example of a function $u$ for which such a function $\kappa(x)$
exists is (in a convex domain $\Om$, see~\cite{Mercier-regularity})
a solution of the so-called ``Rudin-Osher-Fatemi'' functional
\begin{equation}\label{eq:ROF}
\min_u \int_\Om |Du| + \frac{\lambda}{2}\int_\Om (u(x)-f(x))^2dx
\end{equation}
for some $\lambda>0$, in case $f$ is continuous. Then, also $u$
is continuous and $\kappa(x)$ is given by $\lambda(u-f)$. In that
case, in addition, $\kappa(x)$ is also the curvature of the level sets
of $u'=h(u)$ for any nondecreasing function $h$ such that $h(u)$ is still
in $BV(\Om)$.}
\end{remark}
% \begin{remark}\label{rem:uC2} \textup{The proof below can be extended to show that
% if $u\in BV(\Omega)$ is a function with $C^2$ level sets such
% that there exists a continuous function $\kappa(x)$ which coincides
% with the curvature of $\partial\{ u>s\}$ for all $s\in\R$,
% then
% \begin{equation}\label{eq:goodrelaxu}
% F(u) =\int_{\R} \int_{\partial\{u>s\}} f(\kappa_{\partial\{u>s\}}(x))d\H^1(x)
% = \int_\Om f(\kappa(x))d|Du|(x).
% \end{equation}
% It is probably still the case if a.e.~levels of $u$ are $C^2$,
% but this is unclear.
% A typical example of a function $u$ for which such a function $\kappa(x)$
% exists is (in a convex domain $\Om$, see~\cite{Mercier-regularity})
% a solution of the so-called ``Rudin-Osher-Fatemi'' functional
% \begin{equation}\label{eq:ROF}
% \min_u \int_\Om |Du| + \frac{\lambda}{2}\int_\Om (u(x)-f(x))^2dx
% \end{equation}
% for some $\lambda>0$, in case $f$ is continuous. Then, also $u$
% is continuous and $\kappa(x)$ is given by $\lambda(u-f)$. In that
% case, $\kappa(x)$ is also the curvature of the level sets
% of $u'=h(u)$ for any nondecreasing function $h$ such that $h(u)$ is still
% in $BV(\Om)$, so that~\eqref{eq:goodrelaxu} also holds for such a
% $u'$. See Section~\ref{sec:uC2}. }
% \end{remark}

\section{Some properties of the functional}\label{sec:prop}

\subsection{Approximation by smooth functions}
\begin{proposition}\label{thm:smoothapprox}
Assume $\Om$ is a bounded convex set.
Then for any $u\in L^1(\Om)$
with $F(u)<\infty$, there exists $(u_n)_n$
a sequence of functions with $u_n\in C^\infty(\overline\Om)$,
 which converge to $u$ in $L^1(\Om)$, and such that
\[
\lim_{n\to\infty} F(u_n)=F(u).
\]
\end{proposition}
\begin{remark}
\textup{We believe that, upon replacing $C^\infty(\overline\Om)$
with $C^\infty(\Om)$, the result should be true in any domain.
%This is the subject for future study.
}
\end{remark}
\begin{proof}
We assume, without loss of generality, that $0$ is in the
interior of $\Om$. It follows that for any $t<1$, $t\Om\subset\subset\Om$.

% We admit the following fact: for any $\e>0$, there exists
% $\Om^\e\supseteq \{x:\dist(x,\Om)<\e\}$ and a smooth diffeomorphism
% $X_\e:\Om\to\Om^\e$, with $|X_\e(x)-x|\le 5\e$ such that $X_\e(x)\to x$
% in $C^\infty(\Om;\R^2)$ as $\e\to 0$. This is built by locally translating
% $\Om$ outwards near the boundary and stitching the translations.
% We denote $\Xi_\e=(X_\e,\theta)$, which is a diffeomorphism
% from $\Om^\e\times\Sp^1$ to $\OS$. Observe that
% \[
% \nabla\Xi_\e(x,\theta)=\begin{pmatrix}
% \nabla X_\e(x) & \begin{matrix} 0 \\ 0 \end{matrix} \\
% \begin{matrix} 0 & 0 \end{matrix} & 0  \end{pmatrix}.
% \]

% Given then $u$ and an admissible $\sigma$ such that
% $F(u)=\int_{\OS}h(\theta,\sigma)$, we can define
% $\sigma_\e$ a measure on $\Om^\e\times \Sp^1$ by letting, for
% any $\psi\in C_c^0(\Om^\e\times\Sp^1)$,
% \[
% \int_{\Om^\e\times \Sp^1} \psi(x,\theta)d\sigma_\e(x,\theta):=
% \int_{\OS} \nabla \Xi_\e(x)\psi(\Xi_\e(x,\theta)) d\sigma(x,\theta).
% \]
% Then by construction, if $\phi\in C_c^1(\Om^\e\times\Sp^1)$,
% \[
% \int_{\Om^\e\times \Sp^1} \nabla\phi \cdot\sigma_\e=
% \int_{\OS} \nabla \Xi_\e(x)\nabla\phi(\Xi_\e(x,\theta)) d\sigma(x,\theta)
% = \int_{\OS} \nabla (\phi\circ \Xi_\e)\cdot\sigma = 0
% \]
% since $\Div\sigma=0$. Hence $\sigma_\e$ has also free divergence.
%%% HOWEVER THIS IS NOT ADMISSIBLE

Consider $\rho$ a rotationally symmetric mollifier in $\R^2$
($\rho\in C_c^\infty(B_1^2;\R_+)$, $\int\rho=1$, and we let $\rho_\e=(1/\e^2)\rho(x/\e)$,
where here $B_1^d$ is the unit $d$-dimensional ball).
For $\e>0$, we consider the largest $t_\e<1$ such that  
$\{x\in\Om:\dist(x,\partial\Om)>\e\}\supseteq t_\e\Om$.
We then define $\sigma_\e$ as the measure on $\OS$ given by
\[
\int_{\OS} \psi\cdot\sigma_\e  = 
\int_{\OS} \psi*\rho_\e(t_\e x,\theta)d\sigma(x),
\]
where the convolution is only in the $x$-variable:
\[
\psi*\rho_\e(x,\theta) = \int_{B_1^2} \psi(x-\e z,\theta)\rho(z)dz
\]
and is well defined for $x\in t_\e\Om$. It is clear
that $\sigma_\e$ still has free divergence, moreover,
if $\vp\in C_c(\Om;\R^2)$,
\begin{multline*}
\int_{\OS} \vp^\perp\cdot\sigma_\e^x  = 
\int_{\OS} \rho_\e*\vp^\perp(t_\e x)\cdot\sigma_\e^x(x)  
= 
\int_\Om \rho_\e*\vp(t_\e x)\cdot Du
\\= -\int_\Om \int_{B_1^2} t_\e u(x)\rho(z)\Div \vp(t_\e x-\e z) dz dx
\\= -\frac{1}{t_\e}\int_\Om \int_{B_1^2} u((y+\e z)/t_\e)\rho(z)\Div \vp( y) dz dy
 = -\frac{1}{t_\e}\int_\Om (\rho_{\e/t_\e}*u)(y/t_\e)\Div \vp( y) dy.
\end{multline*}
Hence, $\sigma_\e$ is admissible for the function
$u_\e: y\mapsto\rho_{\e/t_\e}*u(y/t_\e)/t_\e$, which clearly goes to $u$
in $L^1(\Om)$ (or $BV(\Om)$ weakly-$*$) as $\e\to 0$. Moreover,
$u_\e\in C^\infty(\Om)$ for all $\e>0$.
Eventually, one has that for any $\vp\in K$,
\begin{equation*}
\int_{\OS} \vp\cdot\sigma_\e= 
\int_{\OS}  \vp*\rho_\e(t_\e x,\theta)\cdot\sigma
\le \int_{\OS} h(\theta,\sigma) = F(u)
\end{equation*}
by observing that since the convolution is only in  the 
$x$ variable, $\vp*\rho_\e(t_\e x,\theta)\in K(\theta)$ for all $(x,\theta)$
when $\vp\in K$.
Thanks to~\eqref{eq:intdualh}, it follows that
\[ 
\int_{\OS} h(\theta,\sigma_\e)\le  F(u).
\]
Hence $F(u_\e)\le F(u)$, and the proposition is proved.
\end{proof}

\subsection{Dual representation}\label{sec:dual}

We can show the following representation for $F$, which in particular
establishes that $F$ coincides exactly with the (more complicated)
 relaxation $R^{**}_\textup{el}$ defined in~\cite{BrPoWi15}.
\begin{proposition} \label{lem:dualform}
Assume $\partial\Om$ is connected.
Then the functional~\eqref{eq:deFu} can also be represented
in the following form
\begin{multline}\label{eq:dualform}
F(u) = \sup\Bigg\{ \int_\Om \psi\cdot Du^\perp\,:\,
\psi\in C^0_c(\Om;\R^2), 
\\
\exists\, \vp\in C_c^1(\OS),
 \utheta\cdot\left(\nabla_x\vp+\psi\right) + f^*(\partial_\theta\vp)\le 0\Bigg\}.
\end{multline}
If $\partial\Om$ is not connected, then same formula
holds, however the functions $\vp$ 
should be in $C^1(\OS)$, with $\nabla\vp\in C_c^0(\OS)$.
\end{proposition}

\begin{corollary}
The functional $F(u)$ coincides with the relaxation of the Elastica
energy $R^{**}_{\textup{el}}(u)$ defined in~\cite{BrPoWi15}.
\end{corollary}
Indeed, formula~\eqref{eq:dualform} is the same as~(35) in \cite[Theorem~13]{BrPoWi15}.
The primal formulation~\eqref{eq:deFu} we introduce here is quite simpler
as the original formulation in~\cite{BrPoWi15}, though, as the ``$\kappa$''
variable is implicit in our formulation (it comes naturally as the
derivative of the orientation) and does not need to be lifted. 

\begin{proof}[Proof of Prop.~\ref{lem:dualform}]
We use the standard perturbation approach to duality~\cite{EkelandTemam},
in the duality $(C_0^0(\OS;\R^3),\M^1(\OS;\R^3))$, where
$C_0^0$ denotes the $L^\infty$-closure of the functions with compact
support (hence, the functions which vanish on the boundary) and $\M^1$
the totally bounded vector-valued Radon measures. Defining, 
for $p\in \M^1(\OS;\R^3)$, the function
\begin{equation}\label{eq:defGp}
G(p)= \inf\left\{ \int_{\OS} h(\theta,\sigma+p)\,:\,
\Div \sigma=0, \int_{\Sp^1}\sigma^x  = Du^\perp \right\}
\end{equation}
so that $F(u)=G(0)$, we observe that $G$ is lower semicontinuous.
Indeed if $p_n\wtos p$ and $\sigma_n$ is a minimizer in the
definition~\eqref{eq:defGp} of $G(p_n)$, 
as before thanks to~\eqref{eq:controlbelowh}, $\sigma_n+p_n$ is
bounded and up to a subsequence, converges to a measure
$\sigma+p$ where $\sigma$ is admissible in~\eqref{eq:defGp}.
It follows that
\[
G(p)\le \int_{\OS}h(\theta,\sigma+p)\le\liminf_n
\int_{\OS}h(\theta,\sigma_n+p_n) = \liminf_n G(p_n).
\]

In particular, we deduce that $G=G^{**}$ so that
\[
F(u) = G(0) = \sup_{\eta\in C_0^0(\OS;\R^3)} - G^*(\eta).
\]
It remains to compute $G^*(\eta)$, to show that~\eqref{eq:dualform}
follows:
\[ %\begin{multline*}
G^*(\eta)=\sup_{p,\sigma} \int_{\OS} \eta\cdot p - h(\theta,\sigma+p)
%\\
= \sup_{\sigma} - \int_{\OS} \eta\cdot\sigma + \sup_p \eta\cdot(\sigma+p)-
h(\theta,\sigma+p)
\] %\end{multline*}
and we find that $G^*(\eta)$ is infinite unless $\utheta\cdot\eta^x+f^*(\eta^\theta)\le 0$ everywhere in $\OS$.
In this case, we let now $\psi(x) = \frac{1}{2\pi}\int_{\Sp^1}\eta^x(x,\theta)d\theta$. Recalling that the sup is over the $\sigma$ admissible,
we find that
\begin{equation}\label{eq:gstar}
G^*(\eta)=\sup_{\sigma} - \int_{\OS} (\eta-\psi)\cdot\sigma -\int_\Om\psi\cdot Du^\perp.
\end{equation}
Now, given any $\sigma'$ with $\Div \sigma'=0$, we can define $u'$
such that $\int_{\Sp^1} \sigma'={Du'}^\perp$ and let
$\sigma = \sigma' + \frac{1}{2\pi}D(u-u')^\perp\otimes d\theta$, in which case
\begin{multline*}
\int_{\OS} (\eta-\psi)\cdot\sigma = \int_{\OS}(\eta-\psi)\cdot\sigma'
+ \frac{1}{2\pi}\int_{\Sp^1} \int_{\Om} (\eta(x,\theta)-\psi(x))\cdot D(u-u')^\perp d\theta 
\\ =\int_{\OS}(\eta-\psi)\cdot\sigma',
\end{multline*}
hence the first supremum in~\eqref{eq:gstar} is the same as the
sup over all $\sigma'$ with vanishing divergence. Hence, it is
zero or $+\infty$, depending on whether $\eta-\psi$ is a gradient
or not. We find eventually $G^*(\eta)$ is finite only when there
exists $\vp\in C^1(\OS)$ with $\nabla\vp=0$ on $\partial\Om\times\Sp^1$,
such that $\eta=\psi+\nabla\vp$.
The proposition follows.
\end{proof}

\begin{remark}\label{rem:smoothing}\textup{
Consider $(\psi,\vp)$ compactly supported continuous functions as in~\eqref{eq:dualform}.
If $t\in (0,1)$, observe that $(\psi_t,\vp_t):=(t\psi,t\vp)$  satisfies
\[
\utheta\cdot(\nabla_x\vp_t+\psi_t)+f^*(\partial_\theta\vp_t)
\le t\left(\utheta\cdot(\nabla_x\vp+\psi)+f^*(\partial_\theta\vp)\right)
+ (1-t)f^*(0)\le -(1-t)\gamma<0.
\]
Let $\rho\in C_c^\infty(B_1^3)$ be a smooth mollifier, $\rho_\e(x,\theta)=
\e^{-3}\rho(x/\e,\theta/\e)$ and
\[
\psi^\e_t = \psi_t*\rho_\e = \int_{B_1^3}\psi_t(x-\e z)\rho(z,\theta)dz d\theta,
\quad\vp^\e_t = \vp_t*\rho_\e.
\]
This is well defined if $\e$ is small enough, as the functions have compact
support. We have that, for $(x,\theta)\in\OS$, thanks to the
convexity of $f^*$, and denoting as before $\eta=\psi+\nabla \vp$ and
$\eta_t=t\eta$,
\begin{multline*}
\utheta \cdot(\nabla_x\vp^\e_t+\psi^\e_t)+f^*(\partial_\theta\vp^\e_t)
\\\le \rho_\e*(\utheta\cdot \eta_t+f^*(\partial_\theta\vp_t))
+ \utheta\cdot\rho_\e*\eta_t-\rho_\e*(\utheta\cdot\eta_t)
\\
\le
-(1-t)\gamma + \int_{B_1^3} \rho(z,\theta')( (\utheta-(\underline{\theta-\e\theta'})) \cdot \eta_t(x-\e z,\theta-\e\theta')).
\end{multline*}
Observing that $\|\utheta-(\underline{\theta-\e\theta'})\|\le \e|\theta'|$,
we deduce that the above is less than $-(1-t)\gamma +t\e\|\eta\|_\infty<0$
as soon as $\e$ is small enough. It follows that the functions
$\psi,\vp$ in~\eqref{eq:dualform} can be assumed to be in $C_c^\infty$.
}
\end{remark}

\section{Numerical Experiments}\label{sec:numexp}

For numerical solution of the proposed model, we need to discretize
both the 2D image domain as well as the 3D domain of the
roto-translation space. Due to the high anisotropy of the energy, one
has to be extremely careful in the choice of the discretization scheme
in order to preserve a maximal degree of rotational invariance while
keeping the numerical diffusion as low as possible. It turns out that
a 2D-3D staggered grid version of an averaged first-order
Raviart-Thomas divergence conforming
discretization~\cite{RaviartThomas1977} yields the best results. More
elaborate discretizations using for example adaptive grids or
higher-order approximations will be subject for future study.

\subsection{Staggered averaged Raviart-Thomas discretization}

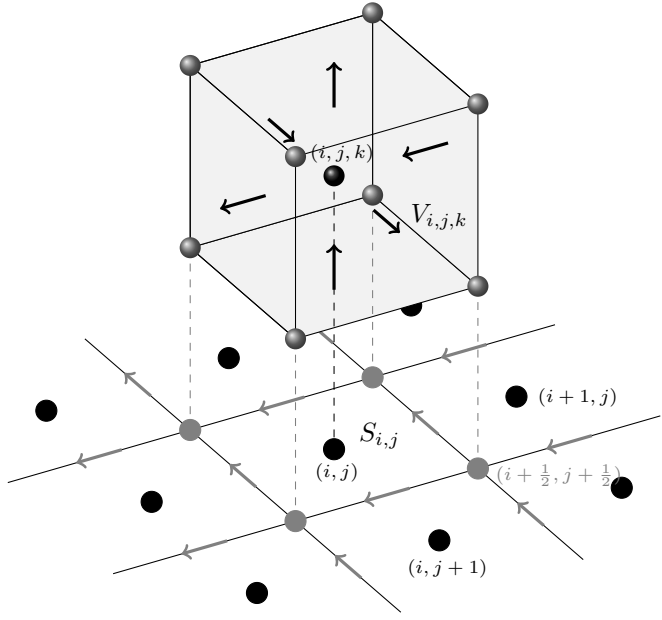
\begin{figure}
  \begin{center}
    \vspace*{-2cm}
    \begin{tikzpicture}
      
      \tdplotsetmaincoords{60}{150}
      \tdplotsetrotatedcoords{0}{0}{0}
      \begin{scope}[tdplot_rotated_coords, shift={(0,0,0)}, scale=1.4]
        
      \foreach \y in {2,4}
      {
        \draw[black] (0,\y) -- (6,\y);
      }

      \foreach \x in {2,4}
      {
        \draw[black] (\x,0) -- (\x,6);
      }

      \foreach \y in {1,3,5}
      \foreach \x in {1,3,5}
      {
        \draw[black,fill=black] (\x,\y) circle (0.1cm);
      }
      
      \foreach \y in {2,4}
      \foreach \x in {2,4}
      {
        \draw[gray,fill=gray] (\x,\y) circle (0.1cm);
      }

      \foreach \y in {2,4}
      \foreach \x in {1-0.25,3-0.25,5-0.25}
      {
        \draw[gray, very thick,->] (\x,\y) -- (\x+0.5,\y);
      }
      \foreach \y in {1+0.25,3+0.25,5+0.25}
      \foreach \x in {2,4}
      {
        \draw[gray, very thick,->] (\x,\y) -- (\x,\y-0.5);
      }
    
      \node at (3+0.2,3+0.4) {\scriptsize $(i,j)$};      
      \node at (1.4,4.5) {\scriptsize $\color{gray} (i+\ha,j+\ha)$};
      \node at (0.5,3.3) {\scriptsize $(i+1,j)$};
      \node at (3+0.2,5+0.5) {\scriptsize $(i,j+1)$};
      \node at (2.5,3) {$S_{i,j}$};
      \end{scope}
      
      \tdplotsetmaincoords{60}{150}
      \tdplotsetrotatedcoords{0}{0}{0}
      \begin{scope}[tdplot_rotated_coords, shift={(2.8,2.8, 2.8)}, scale=1.4]

        % bottom+back
        \draw[black,fill=black!5] (0,0,0) -- (0,2,0) -- (2,2,0) -- (2,0,0) -- cycle;
        \draw[black,fill=black!5] (0,0,0) -- (0,2,0) -- (0,2,2) -- (0,0,2) -- cycle;

        % left
        \draw[black,fill=black!5] (0,0,0) -- (0,0,2) -- (2,0,2) -- (2,0,0) -- cycle;

        % back+front
        \draw[black] (0,0,2) -- (0,2,2) -- (2,2,2) -- (2,0,2) -- cycle;
        \draw[black] (2,0,0) -- (2,2,0) -- (2,2,2) -- (2,0,2) -- cycle;

        \draw[dashed,black] (1,1,1) -- (1,1,-2);
        \draw[dashed,gray] (0,0,0) -- (0,0,-2);
        \draw[dashed,gray] (0,2,0) -- (0,2,-2);
        \draw[dashed,gray] (2,0,0) -- (2,0,-2);
        \draw[dashed,gray] (2,2,0) -- (2,2,-2);

        \shade[ball color = gray] (0,0,0) circle (0.1cm);
        \shade[ball color = gray] (2,0,0) circle (0.1cm);
        \shade[ball color = gray] (0,2,0) circle (0.1cm);
        \shade[ball color = gray] (2,2,0) circle (0.1cm);
        \shade[ball color = gray] (0,0,2) circle (0.1cm);
        \shade[ball color = gray] (2,0,2) circle (0.1cm);
        \shade[ball color = gray] (0,2,2) circle (0.1cm);
        \shade[ball color = gray] (2,2,2) circle (0.1cm);

        % flow through the cube
        \draw[black, very thick,->] (1,1,-0.25) -- (1,1,0.25);
        \draw[black, very thick,->] (1,1,1.75) -- (1,1,2.25);
        \draw[black, very thick,->] (1,-0.25,1) -- (1,0.25,1);
        \draw[black, very thick,->] (1,1.75,1) -- (1,2.25,1);
        \draw[black, very thick,->] (-0.25,1,1) -- (0.25,1,1);
        \draw[black, very thick,->] (1.75,1,1) -- (2.25,1,1);
        
        \node at (1.2,1.5,1.55) {\scriptsize $(i,j,k)$};
        \node at (1,3,1.55) {$V_{i,j,k}$};
        
        % center of cube
        \shade[ball color = black] (1,1,1) circle (0.1cm);

      \end{scope}
      
    \end{tikzpicture}

  \end{center}\caption{Spatial discretization of the image domain $\Om$ and the
    domain of the RT space $\OS$. The bottom grid represents the
    %square pixels $S_{i,j}$ of the
    2D image, and the image intensities $u$ are stored at the vertices
    of the squares $S_{i,j}$.
    The gray arrows indicate
    differences between the pixels. The volume $V_{i,j,k}$ represents
    one element of the corresponding discretization of the RT
    space. The black arrows through the faces of the volume represent
    the 3D vector field $\sigma$.}\label{fig:grid}
\end{figure}

Figure~\ref{fig:grid} shows a visualization of our combined 2D-3D
staggered grid approach. The first grid is given by a 2D grid of
%square 
pixels discretizing the image domain $\Om$. The second grid is
given by a 3D grid of volumes which discretize the roto-translation
space $\OS$. In the visualization we show a few image pixels of the 2D
grid (gray nodes) and one cube of the corresponding 3D grid.

First, we start by describing the discretization of the image
domain. We assume here that $\Om$ is a square or rectangle, or in
general a convex domain, which we will discretize at a scale
$\dx>0$. We will restrict ourselves to the case of a rectangle to
simplify the notation. Then, we let $\Omega_{\dx}$ denote the
discretized image domain which is given by
\[
\Omega_{\dx} = \mathrm{int}\left( \bigcup\Big\{ S_{i,j}\,:
(i,j)\in\Z^2, \, S_{i,j} \subset\Om \Big\} \right),
\]
where
\begin{equation}\label{defSij}
S_{i,j} = \left[(i-\12)\dx, (i+\12)\dx) \right)\times \left[(j-\12)\dx,
    (j+\12)\dx\right),
\end{equation}
denotes the square centered around the index $\i=(i,j)$. 
The intensity
values of the discrete image are stored at the vertices half-grid
points (such as the four vertices of $S_{i,j}$),
and for notational simplicity we introduce the following
index set:
\[
\I = \left\{(i-\12,j-\12) : 1 \leq i \leq N_1, \; 1 \leq j
\leq N_2\right\},
\]
with $N_1, N_2 > 0$. Hence, the discrete image $u \in
\R^{\I}$ consists of $N_1\times N_2$ discrete pixels with
values $(u_{\i})_{\i\in \I}$.  We will later also need
the following index sets:
\begin{align*}
\I^1 =& \left\{(i-\12,j) : 1 \leq i \leq N_1, \;
1 \leq j < N_2\right\},\\ \I^2 =& \left\{(i,j-\12) : 1 < i \leq N_1, \; 1 \leq j \leq N_2\right\},
\end{align*}
which index the middle points of the edges between two adjacent pixels
along the first and second spatial dimensions.

Next, we let $\dt>0$ be the discretization step for the angular
variable. We must assume that $\dt = 2\pi/N_\theta$ with $N_\theta >
0$. Hence $\Sp^1$ is discretized by $N_\theta$ intervals $[(k-\12)\dt,
  (k+\12)\dt)$ for $1\leq k \leq N_\theta$. We then denote by
  $\Gamma_{\delta}$ with $\delta=(\dx,\dt)$ the discretized 3D
  roto-translation space. It is given by
\[
\Gamma_{\delta} = \mathrm{int}\left( \bigcup\Big\{ V_{i,j,k}\,:
(i,j,k)\in\Z^3, \, V_{i,j,k} \subset \OS \Big\} \right),
\]
where
\begin{equation}\label{eq:defVijk}
V_{i,j,k} = S_{i,j} \times \left[(k-\12)\dt,(k+\12)\dt\right)
\end{equation}
denotes the volume centered around the 3D index $\j=(i,j,k)$. We associate
an index set $\J$ of center locations of the volumes
$V_{i,j,k}$:
\begin{equation}\label{eq:defJ}
\J = \left\{(i,j,k) : 1 \leq i < N_1, \; 1 \leq j <
N_2, \; 1 \leq k \leq N_\theta\right\}.
\end{equation}
We shall also introduce the following three index sets:
\begin{align*}
  \J^1 =& \{(i-\12,j,k): 1\leq i \leq N_1, 1\leq j < N_2, 1 \leq k \leq N_\theta\},\\
  \J^2 =& \{(i,j-\12,k): 1\leq i < N_1, 1\leq j \leq N_2, 1 \leq k \leq N_\theta\},\\
  \J^\theta =& \{(i,j,k-\12): 1\leq i < N_1, 1\leq j < N_2, 1 \leq k \leq N_\theta\},
\end{align*}
which we will use to index the facets of the volumes. Observe that the
3D points $((i-\ha)\dx,j\dx,k\dt)$ for all $(i-\ha,j,k) \in
\J^1$ correspond to the middle points of the facets
\begin{equation}\label{eq:defF1}
  \F^1_{(i-\ha,j,k)}=\{(i-\12)\dx\}\times [(j-\12)\dx, (j+\12)\dx) \times [(k-\12)\dt, (k+\12)\dt),
\end{equation}
which are orthogonal to the first spatial
dimension. Similarly, the 3D points $(i\dx,(j-\ha)\dx,k\dt)$ for
all $(i,j-\12,k) \in \J^2$ correspond to the middle
points of the facets
\begin{equation}\label{eq:defF2}
  \F^2_{(i,j-\ha,k)} = [(i-\12)\dx, (i+\12)\dx) \times \{(j-\12)\dx\}\times [(k-\12)\dt, (k+\12)\dt),
\end{equation}
which are orthogonal to the second spatial dimension. Finally, the 3D
points $(i\dx,j\dx,(k-\ha)\dt)$ for all $(i,j,k-\12) \in
\J^\theta$ correspond to the middle points of the facets
\begin{equation}\label{eq:defFt}
  \F^\theta_{(i,j,k-\ha)} = [(i-\12)\dx, (i+\12)\dx) \times [(j-\12)\dx, (j+\12)\dx) \times \{(k-\12)\dt\},
\end{equation}
which are orthogonal to the angular dimension.

Now, we consider a discrete vector field $\sigma=(\sigma^1, \sigma^2,
\sigma^\theta)$ with $\sigma^1 \in \R^{\J^1}$, $\sigma^2 \in
\R^{\J^2}$, and $\sigma^\theta \in
\R^{\J^\theta}$. The values
$(\sigma^1_{\j})_{\j \in \J^1}$,
$(\sigma^2_{\j})_{\j \in \J^2}$,
$(\sigma^\theta_{\j})_{\j\in \J^\theta}$
define the average fluxes through the facets
$(\F^1_{\j})_{\j \in \J^1}$,
$(\F^2_{\j})_{\j \in \J^2}$ and
$(\F^\theta_{\j})_{\j \in
  \J^\theta}$, respectively.

The discrete vector field $\sigma$ can also be seen as a
Raviart-Thomas vector field~\cite{RaviartThomas1977}, defined
everywhere in $\Gamma_{\dx,\dt}$. It is obtained by an affine
extension:
\begin{multline}\label{eq:RaTho}
\sigma_{RT}(x_1,x_2,\theta) = \left(
\sum_{\j = (i-\ha,j,k)\in \J^1} \sigma^1_{\j} \Delta\left(\frac{x_1-(i-\ha)\dx}{\dx}\right), \right.\\\left.
\sum_{\j = (i,j-\ha,k)\in \J^2} \sigma^2_{\j} \Delta\left(\frac{x_2-(j-\ha)\dx}{\dx}\right),
\sum_{\j = (i,j,k-\ha)\in \J^\theta} \sigma^\theta_{\j} \Delta\left(\frac{\theta-(k-\ha)\dt}{\dt}\right)
  \right),
\end{multline}
where $\Delta(t) = \max(0,1-|t|)$ is the usual linear interpolation
kernel.

It is well know that such a Raviart-Thomas field has a divergence
(defined everywhere in the distributional sense)
which is obtained in each volume
by summing the finite differences of the average
fluxes through each pair of opposite facets.
In our model, the discrete field $\sigma$
is constrained to be divergence free, hence for each volume
$V_{\j}$, we must impose that:
\begin{equation}\label{eq:freediv}
\frac{\sigma^1_{i+\ha,j,k}-\sigma^1_{i-\ha,j,k}}{\dx}+
\frac{\sigma^2_{i,j+\ha,k}-\sigma^2_{i,j-\ha,k}}{\dx}+
\frac{\sigma^\theta_{i,j,k+\ha}-\sigma^\theta_{i,j,k-\ha}}{\dt} = 0.
\end{equation}
The roto-translation space is periodic in the $\theta$
direction and hence, whenever $k=N_\theta$, we identify the index
$N_\theta+\ha$ with the index $\ha$ in the computation of the last
term of the discrete divergence.

To simplify our presentation, we introduce a discrete divergence
operator $\mathcal{D}:\R^{\J^1 \cup \J^1 \cup
  \J^\theta}\rightarrow \R^{\J}$, so that the
divergence free constraint~\eqref{eq:freediv} can be compactly written
as
\[
\mathcal{D}\sigma = 0.
\]
Next, to implement a discrete version of the consistency condition we
observe that averaging (in the direction of $\theta$) the first two
components of the field $\sigma$ yields a rotated (by $\pi/2$)
gradient which should locally coincide with the finite differences of
the discrete image $u$. This yields the following discrete consistency
condition:
\begin{equation}\label{eq:compatsigmau}
\begin{cases}
\dt\sum_k \sigma^1_{i+\ha,j,k} = \frac{1}{\dx}
\left(u_{i+\ha,j+\ha}-u_{i+\ha,j-\ha}\right) \\[2mm]
\dt\sum_k \sigma^2_{i,j+\ha,k} = -\frac{1}{\dx}
\left(u_{i+\ha,j+\ha}-u_{i-\ha,j+\ha}\right).
\end{cases}
\end{equation}
We introduce a projection operator $\mathcal{P}:
\R^{\J^1\cup\J^2} \rightarrow \R^{\I^1 \cup
  \I^2}$ and a discrete (rotated) gradient operator
$\mathcal{G}: \R^\I \rightarrow \R^{\I^1 \cup
  \I^2}$ such that we can write the above discrete
compatibility condition as
\[
\mathcal{P}\sigma = \mathcal{G}u.
\]

In order to approximate the continuous energy~\eqref{eq:deFu} with our
discrete Raviart-Thomas field $\sigma_{RT}$ we can %try to
use different types of quadrature rules. After various attempts, we found
out that simply summing the values at the center points of the volumes
$V_{i,j,k}$ provides the highest flexibility for the discrete field
$\sigma$ to concentrate on thin lines, and, in turn, yields the most
faithful numerical results.

Letting thus $\sigma_{i,j,k} = \sigma_{RT}(i\dx, j\dx, k\dt)$, we have that
this discrete field is obtained by the following formula:
\begin{multline}\label{defhatsigma}
  \hat\sigma_{i,j,k} = (\hat\sigma^1_{i,j,k}, \hat\sigma^2_{i,j,k}, \hat\sigma^\theta_{i,j,k}) = \\
  \ha(\sigma^1_{i+\ha,j,k}+\sigma^1_{i-\ha,j,k}, \sigma^2_{i,j+\ha,k}+\sigma^2_{i,j-\ha,k},\sigma^\theta_{i,j,k+\ha}+\sigma^\theta_{i,j,k-\ha}).
\end{multline}
We again introduce an operator $\mathcal{A}:\R^{\J^1 \cup
  \J^2 \cup \J^\theta} \rightarrow
(\R^3)^{\J}$ such that the above averaging operation can be
written
\[
\hat \sigma = \mathcal{A}\sigma.
\]
Now, using the volume centered quadrature rule, the discrete energy 
is
\begin{equation}\label{eq:def-disc-energy}
  F_\delta(u) = \min_{\sigma} \left\{ \dx^2 \dt \sum_{\j=(i,j,k) \in
    \J} h(k\dt, (\mathcal{A}\sigma)_{\j}):
  \mathcal{D} \sigma = 0, \; \mathcal{P}\sigma - \mathcal{G}u =
  0\right\}.
\end{equation}
The consistency of this energy, as $\delta:=(\dx,\dt)\to 0$,
 with the continuous energy $F$ defined
in~\eqref{eq:deFu} is studied in Appendix~\ref{app:consist}.

The regularization function $h(\theta, p)$, where $\theta \in \Sp^1$
and $p = (p^x, p^\theta) \in \R^3$ is defined as in~\eqref{eq:defh},
and the function $f(t)$, which appears in its definition will be one
of the following classical examples of convex functions:
\begin{align}
  f_1(t) &= 1 + \alpha |t|, \label{eq:tac}\tag{TAC}\\
  f_2(t) &= \sqrt{1 + \alpha^2 |t|^2} \label{eq:tgl}\tag{TRV},\\  
  f_3(t) &= 1 + \alpha^2 |t|^2 \label{eq:tsc}\tag{TSC}.
\end{align}
In all examples, the regularizing function is a combination of length
and curvature regularization and the parameter $\alpha > 0$ can be
used to adjust the influence of the curvature regularization.

The first function, $f_1$ is the sum of length and absolute curvature,
hence we call the corresponding regularizer ``total absolute
curvature'' (TAC). One of its main features is that it allows for
sharp corners in the level sets of the image. 
Interestingly, since 
integrating
the absolute curvature along the boundary of a shape is constantly
$2\pi$ for all convex shapes (and in general, is scale independent)
then the main effect of the curvature term in this energy is to
penalize non-convex shapes, regardless of their size.
%Moreover, integrating
%the absolute curvature along the boundary of a shape is constantly
%$2\pi$ for all convex shapes and hence in some sense it favors convex
%shapes.

The function $f_2$ combines length and curvature through an Euclidean
metric and hence, it corresponds to the total variation of the lifted
curve in the RT space, hence we consequently denote this regularizer
``total roto-translational variation'' (TRV). For relatively small
curvature, it favors smooth shapes, but it also allows sharp
discontinuities.

Function $f_3$ penalizes squared curvature plus length and hence is
equivalent to the Elastica energy. We call this regularizer ``total
squared curvature'' (TSC). It is well-known that while length
regularization favors smaller shapes, quadratic curvature favors
larger shapes. Hence, the interplay between length and curvature
regularization removes the shrinkage bias and leads to smooth shapes.

Based on the ``basis'' functions $f_l$, $l=1,2,3$, we obtain the
following convex regularization functions in the RT space.
\begin{eqnarray}\label{eq:defhi}
  \text{(TAC)}\quad h_1(\theta, p) &=&
  \begin{cases}
    |p^x| + \alpha |p^\theta| & \text{if } p^x = \utheta s, \; s \geq 0\\
    +\infty & \text{else,}
  \end{cases}\\
  \text{(TRV)}\quad h_2(\theta, p) &=&
  \begin{cases}
    \sqrt{|p^x|^2 + \alpha^2 |p^\theta|^2} & \text{if } p^x = \utheta s, \; s \geq 0\\
    +\infty & \text{else,}
  \end{cases}\\
  \text{(TSC)}\quad h_3(\theta, p) &=&
  \begin{cases}
    |p^x| + \alpha^2 \frac{|p^\theta|^2}{|p^x|} & \text{if } p^x =
    \utheta s, \; s > 0\\
    +\infty & \text{else.}
  \end{cases}
\end{eqnarray}
Observe in particular, that the above definitions already properly
take into account the correct values of the function $h_l$, $l=1,2,3$
for $p^x = 0$.

In the following sections we will apply the proposed curvature based
regularization functions to a variety of image- and shape processing
problems. For this, we consider generic optimization problems of the
form
\begin{equation}\label{eq:def-disc-problem}
\min_u F_\delta(u) + G(u),
\end{equation}
where $G(u)$ is a convex, lsc. function defining an image-based convex
data fidelity term, possibly dependent on the pixel location.

\subsection{Primal-dual optimization}

In this section, we show how to compute a minimizer of the non-smooth
convex optimization problem~\eqref{eq:def-disc-problem}. For
notational simplicity, we first divide the whole objective
function~\eqref{eq:def-disc-problem} by $\dx^2\dt$ and assume that the
remaining factors are absorbed by the data fitting term.

In order to solve the constrained optimization
problem~\eqref{eq:def-disc-problem}, we consider its Lagrangian
(saddle-point) formulation:
\begin{multline}\label{eq:saddle}
  \min_{u,\sigma} \max_{\phi, \psi, \xi}
  \sum_{\j \in \J} (\mathcal{A}\sigma)_{\j} \cdot \xi_{\j} \hspace{1mm} - \hspace{-4mm}
  \sum_{\j = (i,j,k) \in \J} \hspace{-2mm} h^*(k\dt,\xi_{\j}) +  
  G(u) \hspace{1mm} + \\
  \sum_{\j\in \J} (\mathcal{D}\sigma)_\j\phi_\j + \hspace{-2mm}
  \sum_{\i \in \I^1 \cup \I^2} \hspace{-3mm}\left((\mathcal{P}\sigma)_\i - (\mathcal{G}u)_\i\right) \psi_\i,
\end{multline}
where $\phi \in \R^{\J}$, $\xi = (\xi^1, \xi^2,
\xi^\theta)\in (\R^3)^{\J}$ and $\psi = (\psi^1, \psi^2) \in
\R^{\I^1\cup\I^2}$ with $\psi^1 \in
\R^{\I^1}$ and $\psi^2 \in \R^{\I^2}$ are the dual
variables (Lagrange multipliers or discrete test functions). The
function $h^*$ denotes the convex conjugate of the function
$h$. Recall that $h$ is the support function of the convex set
\[
H(\theta) = \{\xi=(\xi^x,\xi^\theta) \in \R^3: \xi^x\cdot \utheta \leq
-f^*(\xi^\theta) \}, \quad \utheta=(\cos \theta, \sin \theta),
\]
where $f^*$ denotes the convex conjugate of $f$. Hence, the convex
conjugate $h^*$ is simply the indicator function of the set $H$:
\[
h^*(\theta, \xi) = \begin{cases}
  0 & \text{if } \xi \in H(\theta)\\
  \infty & \text{else.}
\end{cases}
\]
The problem~\eqref{eq:saddle} is a saddle-point problem which is
separable in the non-linear terms and hence falls into the class of
problems that can be solved by the first-order-primal-dual algorithm
with diagonal preconditioning and
overrelaxation~\cite{CP2011,PC2011,CP2016}. In order to make the
algorithm implementable, we need efficient algorithms to compute the
projection operators $\mathrm{proj}_{H(\theta)}$ and proximity
operators $\mathrm{prox}_{\tau G}$. They are defined as the
unique minimizers of the minimization problems:
\[
\mathrm{proj}_{H(\theta)}\left(\eta\right) = \arg\min_{\xi\in
  H(\theta)} \frac{1}{2\tau}|\xi-\eta|^2, \; \forall \theta \in \Sp^1,
\]
and,
\[
\mathrm{prox}_{\tau G}\left(v\right) = \arg\min_u G(u) +
\frac{1}{2\tau}\norm{u-v}^2,
\]
for some $\tau > 0$. In what follows, we will detail the projection
operator for different instances of the convex sets $H$ (respectively
regularization functions $h$), the proximity maps for $G$ will be
detailed as soon as they are needed in the numerical results.

The general idea for performing the projection $\xi =
\mathrm{proj}_{H(\theta)}(\eta)$ of a point $\eta=(\eta^x, \eta)$ onto
the set $H(\theta)$ is outlined in Figure~\ref{fig:projection}. The
Figure on the left hand side represents the set $H$ in the $\xi^x$
plane, and the Figure on the right hand side shows the ``profile'' of
the set which is obtained by cutting along the
$\xi_\theta^x\times\xi^\theta$ plane.

\begin{figure}[t!]
  \subfigure[$H \cap \{(\xi^x,\xi^\theta): \xi^x\in \R^2, \xi^\theta=0\}$]{
    \begin{tikzpicture}[scale = 0.5]

      \draw  (-5,0) edge[->, thick] (5,0);
      \draw  (0,-5) edge[->, thick] (0,5);      
      \node at (5.5,0) {$\xi^{x_1}$};
      \node at (0,5.5) {$\xi^{x_2}$};

      % rotation = 60 deg
      % (4,1)--(-4,1)--(-4,-4)--(4,-4)
      \fill[opacity=0.9, gray] (1.1340,3.9641)--(-2.8660,-2.9641)--
      (1.4641,-5.4641)--(5.4641,1.4641);      
      \draw  (-2.8660,-2.9641) edge[ultra thick] (1.1340,3.9641);
      \node at (1.25,-4.75) {$H$};

      % theta direction
      \draw  (4.3301,-2.5000) edge[->, thick] (-4.3301,2.5000);
      \node at (-4.7,2.5) {$\xi_\theta^x$};

      % \eta^x = (3,3) unrotated
      \draw[fill=black] (-1.0981, 4.0981) circle (0.1cm);
      \node at (-1.5,4.25) {$\eta^x$};

      % p^x = (3,0) unrotated
      \draw[fill=black] (1.0670,2.8481) circle (0.1cm);
      \node at (1.6,2.8) {$\xi^x$};
      
      \draw  (-1.0981, 4.0981) edge[->, very thick] (1.0670,2.8481);

      % (0, 3)
      \draw[fill=black] (-2.5981,1.5000) circle (0.1cm);
      \node at (-3,1.0) {$\utheta \eta_\theta^x$};
      
      \draw (-1.0981, 4.0981) edge[dashed, thick] (-2.5981,1.5000);

      % (0,0.5)
      \draw[fill=black] (-0.4330,0.2500) circle (0.1cm);
      \node at (-0.5,-0.5) {$\utheta \xi_\theta^x$};
      
      \draw (1.0670,2.8481) edge[dashed, thick] (-0.4330,0.2500);

      \draw  (-2.5981,1.5000) edge[->, very thick] (-0.4330,0.2500);
      
  \end{tikzpicture}
  }\hfill    
  \subfigure[Profile $P$]{
    \begin{tikzpicture}[scale = 0.5]

      \draw  (-5,0) edge[->, thick] (5,0);
      \draw  (0,-5) edge[->, thick] (0,5);      
    
      \node at (5.5,0) {$\xi_\theta^{x}$};
      \node at (0,5.5) {$\xi^{\theta}$};

      \fill [opacity=0.9, gray, domain=-3.5:3.5, variable=\y]
      (1-3.5*3.5/2,3.5)
      -- plot ({1-\y*\y/2}, {\y})
      -- cycle;

      \node at (-4.5,-2.5) {$P$};
      
      \draw[scale=1.0,domain=-3.5:3.5,smooth,variable=\y,black, ultra thick]
      plot ({1-\y*\y/2},{\y});
      
      \node at (1.25,-0.5) {$1$};
      
      % \eta_\theta^x = (3,3.5) unrotated
      \draw[fill=black] (3, 3.5) circle (0.1cm);
      \node at (3,4) {$(\eta_\theta^x,\eta^\theta)$};

      % \xi_\theta^x = (0.5,1) unrotated
      \draw[fill=black] (0.5, 1) circle (0.1cm);
      \node at (2,1) {$(\xi_\theta^x,\xi^\theta)$};
      
      \draw  (3,3.5) edge[->, very thick] (0.5,1);
      
  \end{tikzpicture}}
  \caption{Projection of the point $(\eta^x,\eta^\theta)$ onto the
    convex sets $H$. First, the point $(\eta_\theta^x, \eta^\theta)$
    is projected onto the profile $P$ which yields $(\xi_\theta^x,
    \xi^\theta)$, then $\xi^x$ is computed from $\xi_\theta^x$ via the
    geometric relationship $\xi^x = \eta^x - \utheta(\eta^x_\theta -
    \xi^x_\theta)$.}\label{fig:projection}
\end{figure}
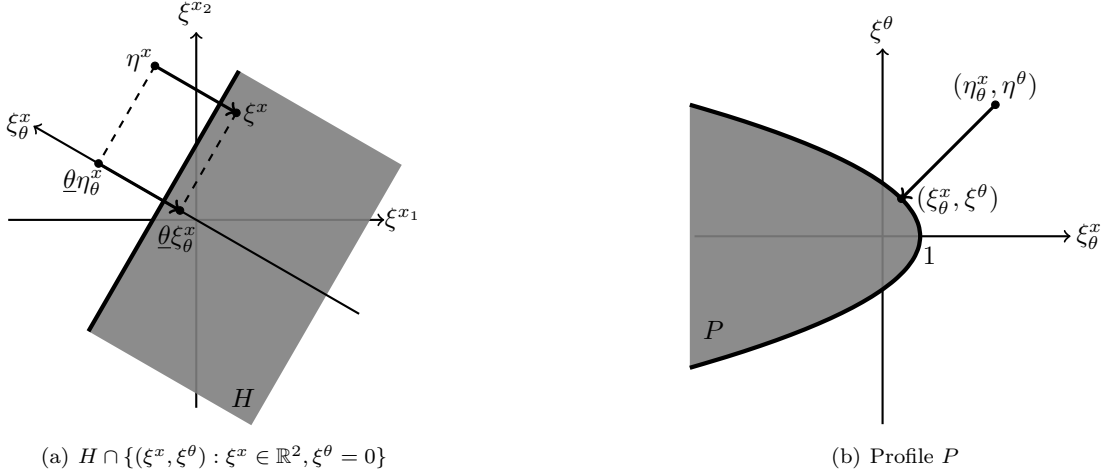

We let $\eta^x_\theta = \eta^x\cdot\utheta$ and define
$(\xi^x_\theta,\xi^\theta) = \mathrm{proj}_{P}(\eta^x_\theta,
\eta^\theta)$ as the projection of the point $(\eta^x_\theta,
\eta^\theta)$ onto the ``profile''
\begin{equation}\label{eq:profile}
P = \{(\xi^x_\theta,\xi^\theta) \in \R^2: \xi^x_\theta \leq
-f^*(\xi^{\theta}) \}.
\end{equation}
Then, using simple geometric reasoning, we see that the variable
$\xi^x$ can be recovered from $\xi^\theta$ by
\[
  \xi^x = \eta^x - \utheta(\eta^x_\theta - \xi^x_\theta).
\]

It remains to detail the projections onto different profiles $P$.
\begin{itemize}
\item TAC: $f_1(t) = 1+\alpha|t|$: The convex conjugate of the
  function $f_1$ is given by
  \[
  f_1^*(s) = \begin{cases}
    -1 & \text{if } |s| \leq \alpha\\
    \infty & \text{else},
  \end{cases}
  \]
  and in turn the profile $P_1$ is given by
  \[
  P_1 = \{(\xi^x_\theta,\xi^\theta) \in \R^2: \xi^x_\theta \leq
  1, \; |\xi^{\theta}| \leq \alpha \}
  \]
  It is straightforward that the projection of a point
  $(\eta^x_\theta, \eta^\theta)$ onto the profile $P_1$ can be
  performed via simple truncation operations
  \[
  \begin{pmatrix}
    \xi^x_\theta\\ 
    \xi^{\theta}
  \end{pmatrix} =
  \begin{pmatrix}    
    \min(1, \eta^x_\theta)\\
    \max(-\alpha, \, \min(\alpha, \eta^\theta))
  \end{pmatrix}.
  \]
\item TRV: $f_2(t) = \sqrt{1+\alpha^2t^2}$: A simple computation shows that
  the convex conjugate of $f_2(t)$ is given
  \[
  f_2^*(s) = \begin{cases}
    -\sqrt{1-s^2/\alpha^2} & \text{if } |s| \leq \alpha\\
    \infty & \text{else.}
  \end{cases}
  \]
  Inserting the expression of the convex conjugate into the
  profile~\eqref{eq:profile} and squaring both sides, we obtain
  \[
  P_2 = \{(\xi^x_\theta,\xi^\theta) \in \R^2: \max(0,\xi^x_\theta)^2 +
  (\xi^\theta/\alpha)^2 \leq 1\}.
  \]
  In what follows, we assume that $\max(0,\eta^x_\theta)^2 +
  (\eta^\theta/\alpha)^2 > 1$ since otherwise we do not need to
  project the point. We first treat the case $\eta^x_\theta \leq
  0$. It is easy to see that in this case the solution of the
  projected point is given by
  \[
  \begin{pmatrix}
    \xi^x_\theta\\ 
    \xi^{\theta}
  \end{pmatrix} =
  \begin{pmatrix}
    \eta_\theta^x\\
        \max(-\alpha, \, \min(\alpha, \eta^\theta))
  \end{pmatrix}.
  \]
  In case $\eta^x_\theta > 0$, computing the projection of a point
  $(\eta_\theta^x, \eta^\theta)$ onto the boundary of $P_2$ amounts to
  solve the following equality constrained optimization problem
  \[
  \min_{(\xi^x_\theta)^2 + (\xi^\theta/\alpha)^2 = 1} \frac12
  |\xi_\theta^x-\eta_\theta^x|^2 + \frac12 |\xi^\theta-\eta^\theta|^2.
  \]
  The Karush-Kuhn-Tucker (KKT) optimality conditions for the above
  problem are given by
  \begin{eqnarray*}
    \xi_\theta^x-\eta_\theta^x + 2\lambda(\xi^x_\theta) &=&
    0,\\ \xi^\theta-\eta^\theta + 2\lambda\xi^\theta/\alpha^2 &=&
    0,\\ (\xi^x_\theta)^2 + (\xi^\theta/\alpha)^2 -1 &=& 0,
  \end{eqnarray*}
  where $\lambda > 0$ is the Lagrange multiplier which is positive
  since the point $(\eta_\theta^x, \eta^\theta)$ was assumed to be
  outside $P_2$. Combining the first three equations shows that the
  Lagrange multiplier is computed from the roots of the fourth-order
  polynomial
  \[
  (\alpha+2\lambda/\alpha)^2(\eta^x_\theta )^2 +
  (1+2\lambda)^2(\eta^\theta)^2 -
  (1+2\lambda)^2(\alpha+2\lambda/\alpha)^2 = 0.
  \]
  Let us observe that in case $\alpha=1$, the solution for $\lambda$
  is particularly simple, indeed
  \[
  \lambda = \frac12(\sqrt{(\eta_\theta^x)^2 + (\eta^\theta)^2}-1),
  \]
  such that from the first two equations we obtain for the projected
  point
  \[
  \begin{pmatrix}
    \xi^x_\theta\\ 
    \xi^{\theta}
  \end{pmatrix} = \frac1{\sqrt{(\eta_\theta^x)^2 + (\eta^\theta)^2}}
  \begin{pmatrix}
    \eta_\theta^x\\
    \eta^\theta
  \end{pmatrix}.
  \]
  In the general case $\alpha \not=1$ we compute $\lambda$ by applying
  Newton's algorithm to find the (correct) root from the fourth order
  polynomial. In our experiments it turns out that a large enough
  initial value, e.g. $\lambda = 10^3$ provides a good initialization
  for Newton's algorithm.  Usually, we need less than 5-10 iterations
  of Newton's algorithm to converge to a solution with feasibility
  error less than $10^{-9}$. From the first two equations of the KKT
  conditions the expression of the projected point is given by:
  \[
  \begin{pmatrix}
    \xi^x_\theta\\ 
    \xi^{\theta}
  \end{pmatrix} =
  \begin{pmatrix}
    \eta_\theta^x/(1+2\lambda)\\
    \eta^\theta/(1+2\lambda/\alpha^2)
  \end{pmatrix}.
  \]

\item TSC: $f_3(t) = 1+\alpha^2t^2$: The convex conjugate of $f_3$ is
  computed as
  \[
  f_3^*(s) = (s/(2\alpha))^2-1,
  \]
  and hence the profile $P_3$ is given by
  \[
  P_3 = \{(\xi^x_\theta,\xi^\theta) \in \R^2: \xi_\theta^x +
  (\xi^\theta/(2\alpha))^2 \leq 1\}.
  \]
  Following the same approach as before a point $(\eta_\theta^x,
  \eta^\theta)$ with $\eta_\theta^x + (\eta^\theta/(2\alpha))^2 > 1$
  is projected onto $P_3$ by solving the KKT optimality
  conditions
  \begin{eqnarray*}
    \xi_\theta^x-\eta_\theta^x + \lambda &=&
    0,\\ \xi^\theta-\eta^\theta + \lambda\xi^\theta/(2\alpha^2) &=&
    0,\\ \xi^x_\theta + (\xi^\theta/(2\alpha))^2 -1 &=& 0.
  \end{eqnarray*}
  Combining the above equations, the Lagrange multiplier is obtained
  by finding the correct root from the third-order polynomial
  \[
  (2\alpha^2 + \lambda)^2(\eta_\theta^x-1-\lambda) + (\alpha\eta^\theta)^2 = 0.
  \]
  We again apply Newton's method and observe that for a large enough
  initial value, e.g. $\lambda = 10^3$, Newton's algorithm rapidly
  converges to the correct root of the polynomial. The projected point
  is finally given by
  \[
  \begin{pmatrix}
    \xi^x_\theta\\ 
    \xi^{\theta}
  \end{pmatrix} =
  \begin{pmatrix}
    \eta_\theta^x - \lambda\\
    \eta^\theta/(1+\lambda/(2\alpha^2))
  \end{pmatrix}.
  \]
\end{itemize}

\subsection{Computing a disk}

\begin{figure}[t!]
  \centering
  \subfigure[$u^0$]{\includegraphics[width=0.19\textwidth]{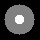}\label{fig:approx-disk-a}}\\
  \subfigure[$N_\theta=4$]{\includegraphics[width=0.19\textwidth]{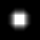}}\hfill
  \subfigure[$N_\theta=8$]{\includegraphics[width=0.19\textwidth]{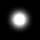}}\hfill
  \subfigure[$N_\theta=16$]{\includegraphics[width=0.19\textwidth]{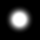}}\hfill
  \subfigure[$N_\theta=32$]{\includegraphics[width=0.19\textwidth]{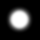}}\hfill
  \subfigure[$N_\theta=64$]{\includegraphics[width=0.19\textwidth]{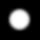}\label{fig:disk-streams}}\\
  \centering \includegraphics[width=0.7\textwidth]{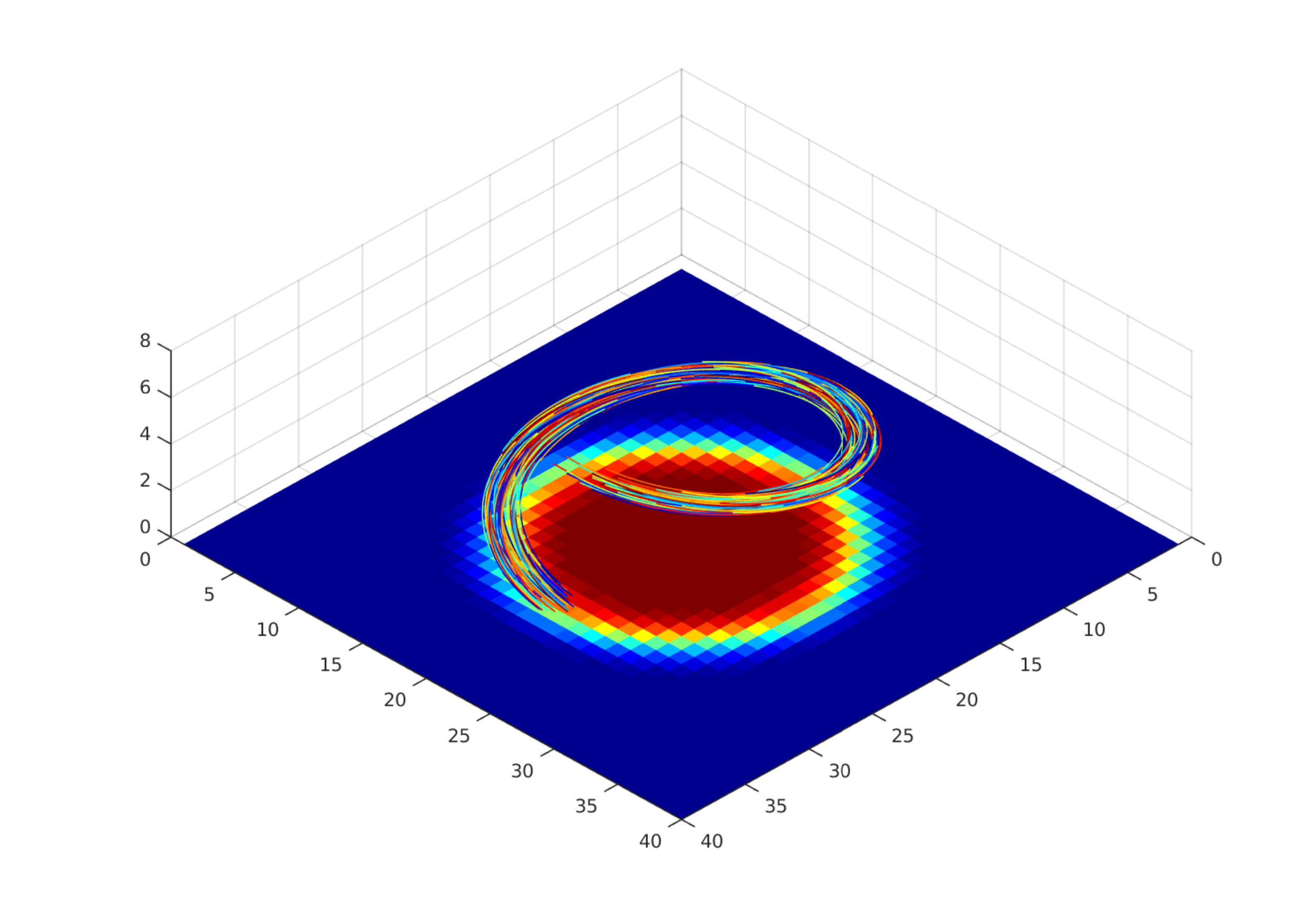}
  \caption{Computing a disk by minimizing the TSC energy. We set
    $\alpha=10$ and hence the optimal solution of the inpainting
    problem is given by a disk of radius $r=10$. (a) shows the input
    image $f$ of size $40 \times 40$ pixels together with its
    inpainting domain indicated by gray pixels and (b)-(e) show the
    computed solution images $u$ for different numbers of discrete
    orientations $N_\theta$. (f) shows a ``stream-line'' plot of the
    averaged Raviart-Thomas vector field $\hat \sigma$ on top of a
    false-color version of the corresponding image
    $u$.}\label{fig:approx-disk}
\end{figure}

In the first example, we consider the most basic numerical experiment,
which is using TSC energy to compute a disk of a given radius. The aim
of this experiment is to investigate the quality of our proposed
discretization scheme, in particular when using a different number of
discrete orientations.  Consider a 2D disk $B(0, r) = \{x \in \R^2:
|x| \leq r\}$ of radius $r > 0$ centered around the origin. The TSC
energy of the boundary of the disk $B(0,r)$ is given by
\[
\int_{\partial B(0,r)} (1+\alpha^2\kappa^2) d\H^1 = 2\pi (r+\alpha^2/r).
\]
Here our disk will be represented by its characteristic function
$\chi_{B(0,r)}(x)$. We can force the minimizer of the TSC energy to
yield a disk using suitable boundary conditions. For example, we can
force at least one point inside the disk to be one and at least one
point outside the disk to be zero. The following simple computation
shows that the minimizer of the TSC energy will be a disk of radius
$r=\alpha$, indeed:
\begin{multline*}
\int_{\partial B(0,r)} (1+\alpha^2\kappa^2) d\H^1 = \\\int_{\partial
  B(0,r)} \left(\alpha^2(\kappa - \frac{1}{\alpha})^2 + 2\alpha\kappa \right) d\H^1 = \int_{\partial
  B(0,r)} \alpha^2(\kappa - \frac{1}{\alpha})^2d\H^1 + 4\pi\alpha,
\end{multline*}
which is minimized if $\kappa=1/r=1/\alpha$ everywhere.

We set up a our optimization problem using a 2D grid of $N_1\times
N_2=40\times 40$ pixels. We will use a different number of discrete
orientations $N_\theta \in \{4,8,16,32,64\}$ in order to investigate
the quality of the approximation depending on $N_\theta$.
In this and the subsequent experiments, the discretization
width of the spatial grid is set to $\dx=1$ and the discretization
width of the angular dimension is set to $\dt=2\pi/N_\theta$. The aim
of this experiment is to compute a disk of radius $r=10$ pixels. We
therefore define an inpainting domain $\mathcal{D} \subseteq \I$
forming a band of 10 pixels width around a disk of radius $r=10$. In
this domain we minimize the TSC energy and we use the remaining part
of the image as boundary condition. The corresponding data term in our
optimization problem~\eqref{eq:def-disc-problem} is given by
\begin{equation}\label{eq:data-term-inpainting}
  G(u) = \sum_{\i \in \mathcal{D}} \iota_{u^0_{\i}}(u_\i),
\end{equation}
where $\iota_C$ denotes the indicator function of the convex set
$C$. The proximal map for $G(u)$ is given by
\begin{equation}\label{eq:prox-inpainting}
  u = \mathrm{prox}_{\tau G}(v) \; \Longleftrightarrow \;
  u_\i =
  \begin{cases}
    v_\i   & \text{if } \i \in \mathcal{D}\\
    u^0_{\i} & \text{else}
  \end{cases},
  \; \forall \i \in \I.
\end{equation}
Figure~\ref{fig:approx-disk-a} shows the input image $(u^0_{\i})_{\i
  \in \I}$ and the inpainting domain $\mathcal{D}$ is indicated by the
gray area.

%K= 4: iter = 500000: tv = 60.10632, tac = 6.34558, W = 1.75038
%K= 8: iter = 500000: tv = 54.80439, tac = 6.28472, W = 0.89295
%K=16: iter = 500000: tv = 58.50410, tac = 6.28740, W = 0.70412
%K=32: iter = 500000: tv = 61.52567, tac = 6.28349, W = 0.64483
%K=64: iter = 500000: tv = 62.93360, tac = 6.28353, W = 0.62768
\begin{table}[ht!]
  \centering
  \begin{tabular}{c|ccc}
    \toprule  
    $N_\theta$ & $H_{\mathrm{TV}}$ $(2 \pi r\approx 62.8319)$ &
    $H_{\mathrm{AC}}$ $(2\pi \approx 6.2832)$ &
    $H_{\mathrm{SC}}$ $(2\pi/r \approx 0.6283)$\\
    \midrule
    4  & 60.1063 & 6.3456 & 1.7504\\
    8  & 54.8043 & 6.2847 & 0.8930\\
    16 & 58.5041 & 6.2874 & 0.7041\\
    32 & 61.5257 & 6.2835 & 0.6448\\
    64 & 62.9336 & 6.2835 & 0.6277\\
    \bottomrule
  \end{tabular}  
  \caption{Approximating a disk of radius $r=10$. The table shows the
    values of the computed total variation (TV), absolute curvature
    (AC), and squared curvature (SC) for a varying number of discrete
    orientations $N_\theta$. In parentheses we also give the values of
    the respective energies for the true solution. Observe that the
    squared curvature is approximated well only when using a quite
    large number of discrete orientations.}
  \label{tab:approx-disk}
\end{table}

In order to quantify the approximation quality of our discretization
for a varying number of discrete orientations, we report the total
variation (TV), the absolute curvature (AC) and the squared curvature
(SC). All three quantities are computed from the averaged field $\hat
\sigma = \mathcal{A}\sigma$ using the volume-centered based discrete
energy
\[
H(\hat\sigma) = \dx^2 \dt \sum_{\j \in \J} h(\hat \sigma_{\j}),
\]
where the function $h$ is one of the three instances:
%\begin{align*}
\[
  h_{\mathrm{TV}}(\hat \sigma_{\j})  = \sqrt{ (\hat \sigma^1_{\j})^2 + (\hat \sigma^2_{\j})^2}, \quad
  h_{\mathrm{AC}}(\hat \sigma_{\j})  =|\hat \sigma^\theta_{\j}|, \quad
  h_{\mathrm{SC}}(\hat \sigma_{\j})  = \frac{|\hat \sigma^\theta_{\j}|^2}{\sqrt{ (\hat \sigma^1_{\j})^2 + (\hat \sigma^2_{\j})^2}}.
\] %\end{align*}
Table~\ref{tab:approx-disk} details the values of the discrete
energies we obtained for different numbers of discrete orientations
$N_\theta$. From the results, one can see that a higher number of
discrete energies generally leads to a better approximation of the
true energy. This is particularly true for the value of the squared
curvature which seems to be well approximated only when using a quite
high number of discrete orientations. The absolute curvature, however
seems to be well approximated even when using only a small number of
discrete orientations. The reason for this is that the absolute
curvature of a smooth curve is easy to approximate by means of
a piecewise linear curve.

Figure~\ref{fig:disk-streams} visualizes the averaged Raviart-Thomas
vector field $\hat \sigma$ of the disk example in case of
$N_\theta=64$ discrete orientations. Observe that the vector field
nicely corresponds to the expected shape of a helix, shown in
Figure~\ref{fig:rtspace}. However, due to diffusive effects of our
numerical scheme, the vector field does not perfectly concentrate on a
one-dimensional structure.

\subsection{Non-smooth level sets}
\begin{figure}[htb]
  \centering \subfigure[Original
    image]{\includegraphics[width=0.32\textwidth]{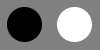}\label{fig:infinity-a}}\hfill
  \subfigure[Input image
    $u^0$]{\includegraphics[width=0.32\textwidth]{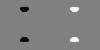}\label{fig:infinity-b}}\hfill
  \subfigure[Computed image
    $u$]{\includegraphics[width=0.32\textwidth]{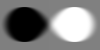}\label{fig:infinity-c}}\\ \subfigure[Visualization
    of the vector
    field]{\includegraphics[width=0.7\textwidth]{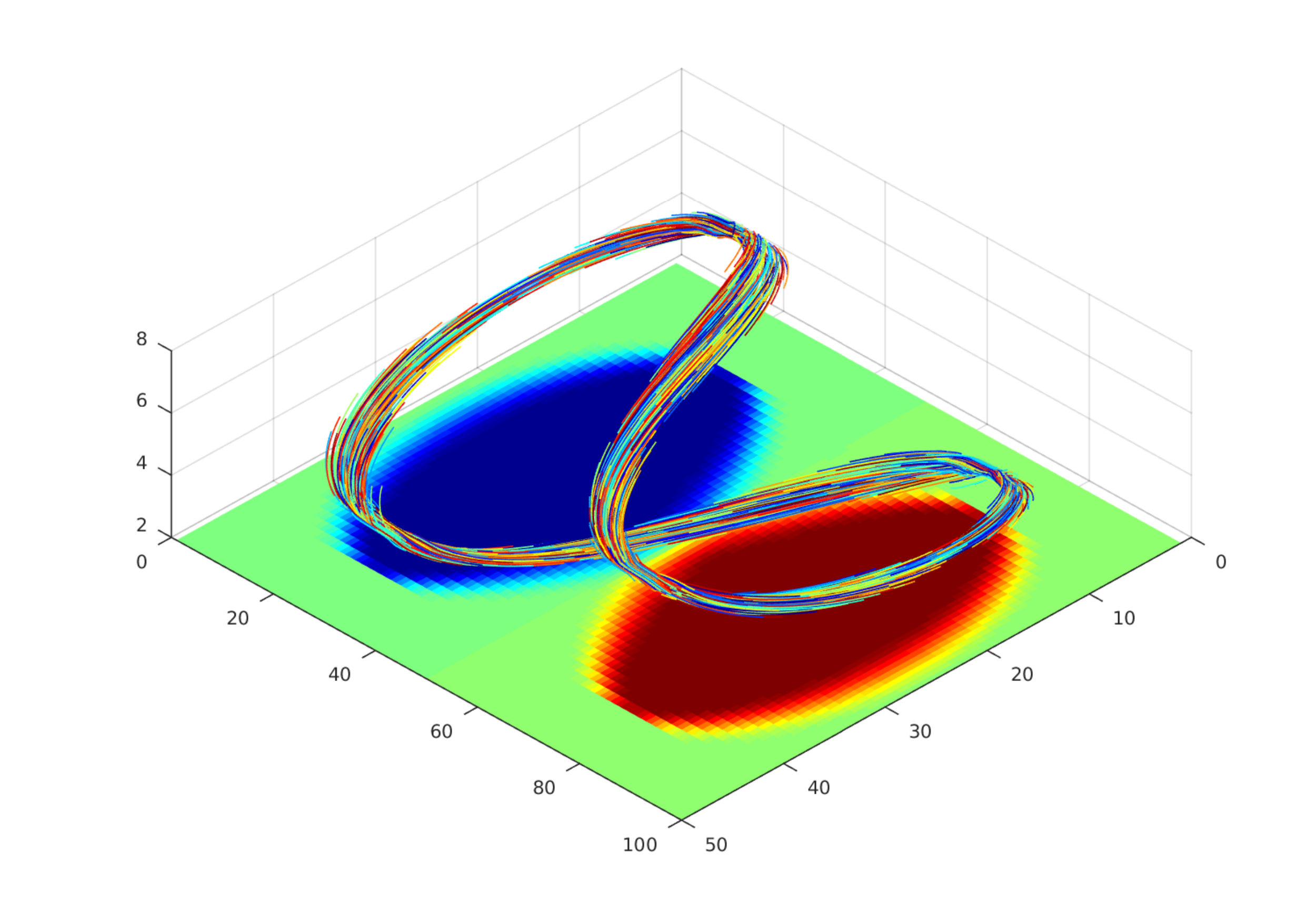}\label{fig:infinity-d}}
  \caption{Effect of convexifying curvature depending energies in the
    roto-translation space. (a) shows the original image, (b) shows
    the input image $u^0$ of size $100\times 50$ pixels, where gray
    pixels indicate the inpainting domain. (c) shows the computed
    solution of minimizing the TSC energy. (d) visualizes the vector
    field $\hat \sigma$ in the roto-translation space. Observe that
    the twisted $\infty$-shape curve skips the strong curvature at the
    cusp.}\label{fig:infinity}
\end{figure}

Here we demonstrate a typical effect of our convexification which
finds a low-energy solution for an image with non-smooth level sets
which should have infinite energy in more standard
relaxations~\cite{MasnouMorel06} of the Elastica
energy. Figure~\ref{fig:infinity-a} shows an image of size $N_1\times
N_2 = 100 \times 50$ pixels of a black and a white disk in front of a
gray background. Similarly to the inpainting problem of the previous
example, we fix only the four small parts of the original image (see
Figure~\ref{fig:infinity-b}) and minimize the TSC energy in the
inpainting domain $\mathcal{D}$, indicated by the gray area. In this
experiment we used $N_\theta=64$ discrete orientations and we set
$\alpha = 17$ to match the radius of the
disks. Figure~\ref{fig:infinity-c} shows the computed minimizer of the
TSC energy. Observe that the solution $u$ does not yield the expected
two disks of the original image but rather drop-like shapes that form
a sharp cusp in the middle of the image. Figure~\ref{fig:infinity-b}
shows a stream-line representation of the minimizing vector field
$\hat \sigma$ in the roto-translation space. Inspecting the field, one
can immediately see the reason for this behavior. The vector field
$\hat \sigma$ forms a twisted $\infty$-shape curve which ``skips'' the
strong curvature of the cusp.
%\clearpage

\subsection{Shape completion}

\begin{figure}[t!]
  \centering
  \subfigure[Input
    image]{\includegraphics[width=0.24\textwidth]{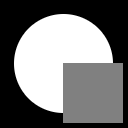}}\hfill
  \subfigure[TAC,
    $\alpha=15$]{\includegraphics[width=0.24\textwidth]{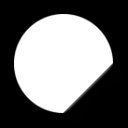}}\hfill
  \subfigure[TRV,
    $\alpha=15$]{\includegraphics[width=0.24\textwidth]{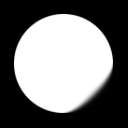}}\hfill
  \subfigure[TSC,
    $\alpha=50$]{\includegraphics[width=0.24\textwidth]{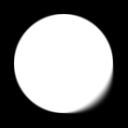}}\\ \subfigure[Input
    image]{\includegraphics[width=0.24\textwidth]{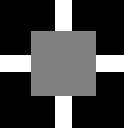}}\hfill
  \subfigure[TAC,
    $\alpha=15$]{\includegraphics[width=0.24\textwidth]{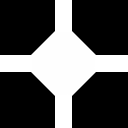}}\hfill
  \subfigure[TRV,
    $\alpha=15$]{\includegraphics[width=0.24\textwidth]{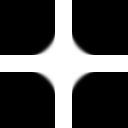}}\hfill
  \subfigure[TSC,
    $\alpha=10$]{\includegraphics[width=0.24\textwidth]{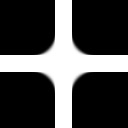}}\\ \subfigure[Input
    image]{\includegraphics[width=0.24\textwidth]{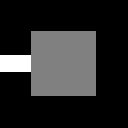}}\hfill
  \subfigure[TAC,
    $\alpha=15$]{\includegraphics[width=0.24\textwidth]{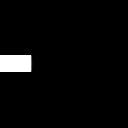}}\hfill
  \subfigure[TRV,
    $\alpha=15$]{\includegraphics[width=0.24\textwidth]{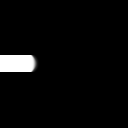}}\hfill
  \subfigure[TSC,
    $\alpha=10$]{\includegraphics[width=0.24\textwidth]{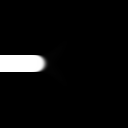}}\\
  \subfigure[Input images with rotations $0,\pi/8,\pi/4$]{
    \includegraphics[width=0.15\textwidth]{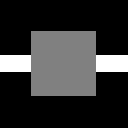}\hspace{0.5mm}
    \includegraphics[width=0.15\textwidth]{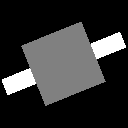}\hspace{0.5mm}
    \includegraphics[width=0.15\textwidth]{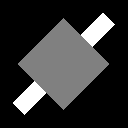}}\hfill
  \subfigure[TAC, $\alpha=15$]{
    \includegraphics[width=0.15\textwidth]{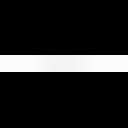}\hspace{0.5mm}
    \includegraphics[width=0.15\textwidth]{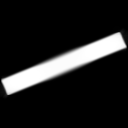}\hspace{0.5mm}
    \includegraphics[width=0.15\textwidth]{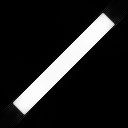}}
  \caption{Shape completion for various shapes using total absolute
    curvature (TAC), total roto-translational variation (TRV) and total squared
    curvature (TSC). The inpainting domain is indicated by the gray
    pixels. In the last row we provide results for completing straight
    lines at different angles.}\label{fig:shape-inpaint}
\end{figure}
\begin{figure}[t!]
  \centering
  \subfigure[Original shape]{\includegraphics[width=0.32\textwidth]{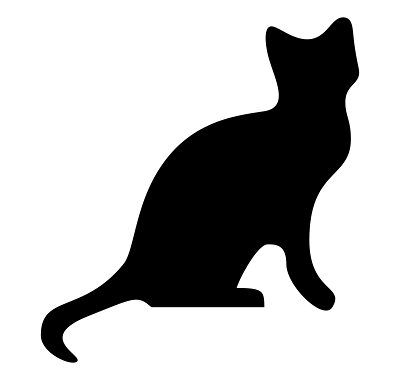}}\hfill
  \subfigure[Dipoles]{\includegraphics[width=0.32\textwidth]{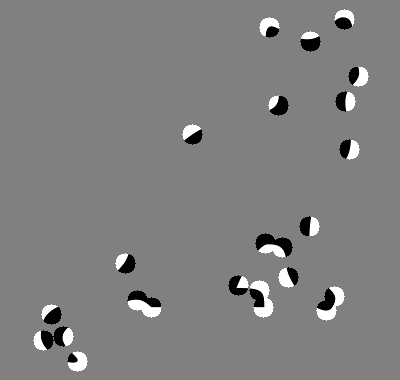}}\hfill
  \subfigure[TSC, $\alpha=50$]{\includegraphics[width=0.32\textwidth]{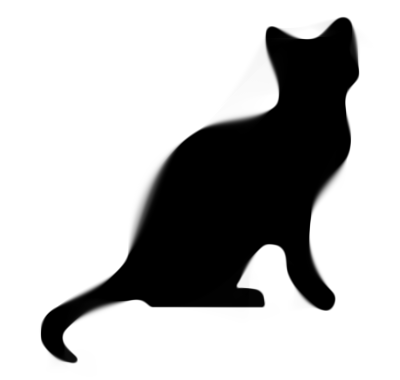}}
  \caption{Computation of ``Weickert's cat''~\cite{WeickertCat2012}:
    (a) shows the original shape, (b) depicts the given dipole data
    and (c) is the result of minimizing the TSC energy in the gray
    inpaiting domain.}\label{fig:weickerts-cat}
\end{figure}

In this section we provide some qualitative results on a number of
different shape completion problems. Similar to the previous two
examples, we define an inpainting domain $\mathcal{D}$ which is
indicated by the gray area and keep the remaining image as boundary
condition. Figure~\ref{fig:shape-inpaint} shows various input shapes
with their inpainting domains and the solutions of minimizing
different curvature energies using different settings of the parameter
$\alpha$. One can see that while minimizing the TAC energy usually
leads to straight connections and sharp corners, TSC leads to a smooth
continuation of the boundaries. The TRV energy leads to results which
are somewhere in between the results of TAC and TSC. In the last row of
Figure~\ref{fig:shape-inpaint}, we additionally demonstrate the
behavior for completing straight lines at different rotations by
minimizing the TAC energy. It turns out that our discretization scheme
performs quite well for rotations of $0, \pi/4, \pi/2, \ldots$ but
leads to more diffusive results for rotations of $\pi/8, 3\pi/8,
\ldots$. The development of a more isotropic scheme will be subject of
future research.

In Figure~\ref{fig:weickerts-cat}, we show the application to shape
completion from dipoles using the original data of ``Weickert's
cat''~\cite{WeickertCat2012}. The input image $u^0$ is of size
$N_1\times N_2 = 400 \times 380$ pixels and the inpainting domain
$\mathcal{D}$ is again indicated by the gray pixels. In this example
we used $N_\theta=64$ discrete orientations and the curvature
parameter was set to $\alpha=50$. In order to avoid relaxation
artifacts (self-intersections) inside the dipoles, we additionally use
zero-boundary conditions for the field $\sigma^{1,2}$ within the
constant areas of the dipoles. The results show that our numerical
scheme can successfully reconstruct the curvilinear shape of the cat
from only very little information given by the dipoles. However, we
can also observe %minor 
diffusion artifacts of our numerical scheme,
especially if the reconstructed boundaries are relatively
long. Sharper results for this kind of problems are usually obtained
using sophisticated anisotropic diffusion schemes such as the edge
enhancing diffusion (EED)~\cite{Weickert1996,Schmaltzetal2014}.
It would be interesting to understand whether these schemes also
% However, for these schemes it is still unknown if they 
minimize an underlying variational energy.

\subsection{Shape regularization}
\begin{figure}[t!]
  \centering
  \subfigure[Bull fight]{\includegraphics[width=0.325\textwidth]{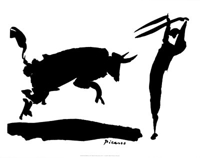}}\\
  \subfigure[TAC, $\lambda=8$]{\includegraphics[width=0.325\textwidth]{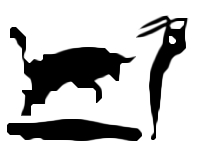}}
  \subfigure[TAC, $\lambda=4$]{\includegraphics[width=0.325\textwidth]{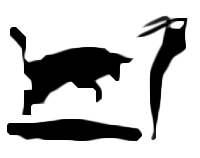}}
  \subfigure[TAC, $\lambda=2$]{\includegraphics[width=0.325\textwidth]{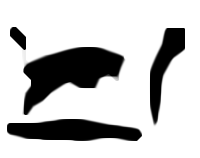}}\\
  \subfigure[TRV, $\lambda=8$]{\includegraphics[width=0.325\textwidth]{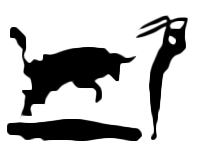}}
  \subfigure[TRV, $\lambda=4$]{\includegraphics[width=0.325\textwidth]{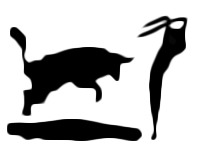}}
  \subfigure[TRV, $\lambda=2$]{\includegraphics[width=0.325\textwidth]{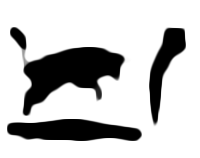}}\\
  \subfigure[TSC, $\lambda=8$]{\includegraphics[width=0.325\textwidth]{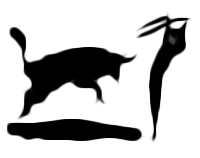}}
  \subfigure[TSC, $\lambda=4$]{\includegraphics[width=0.325\textwidth]{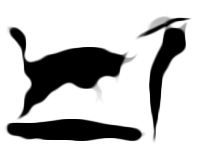}}
  \subfigure[TSC, $\lambda=2$]{\includegraphics[width=0.325\textwidth]{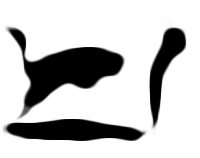}}
  \caption{Shape regularization using curvature regularization. (a)
    shows the original ``Bull fight'' image. (b)-(d) show the result
    of TAC regularization, (e-g) show the result of TRV regularization
    and (h)-(j) show the result of TSC
    regularization.}\label{fig:bull-fight}
\end{figure}

In our next experiment, we apply our curvature based energies for
shape regularization. Given an input image $(u^0_{\i})_{\i\in\I}$,
which is the characteristic function of a given shape, our aim is to
compute a simplified (or regularized) shape which is represented by
means of a (relaxed) binary image $(u_\i)_{\i\in\I}$. We make use of a
simple linear fidelity term which is frequently used in image
segmentation:
\[
G(u) = \sum_{\i \in \I} g_{\i}, \quad g_\i(u_{\i}) = u_{\i}w_\i +
\iota_{[0,1]}(u_\i),
\]
where $w \in \R^\I$ is a force field. For the application to shape
regularization we use $w = \lambda(\ha-u^0)$, where $\lambda >0$
defines the strength of the force field. The proximal map for this
data term is easily computed:
\[
u = \mathrm{prox}_{\tau G}(v) \; \Longleftrightarrow \; u_\i= \max(0, \min(1, v_\i - \tau w_\i)), \; \forall \i \in \I.
\]
In Figure~\ref{fig:bull-fight}, we apply TAC, TRV and TSC
regularization to regularize the shape of Picasso's ``Bull fight''
image. In all three cases, we set $\alpha=10$ and we use different
settings of the parameter $\lambda$ to obtain gradually simplified
shapes. From the results one can see that TAC regularization yields
shapes with relatively straight boundaries and sharp corners. TRV
regularization yields smooth shapes but also allows for sharp corners.
TSC regularization yields smooth shapes. Observe that whenever it
seems energetically preferable, the solution of TSC produces sharp
cusps with ``hidden'' edges to bypass locations of strong
curvature, as predicted by the theory~\cite{Bellettini-MS-1993}.
This effect is usually less visible when minimizing the TAC
or TRV energies.

\subsection{Image inpainting}

\begin{figure}[t!]
  \centering
  \subfigure[Einstein]{\includegraphics[width=0.325\textwidth]{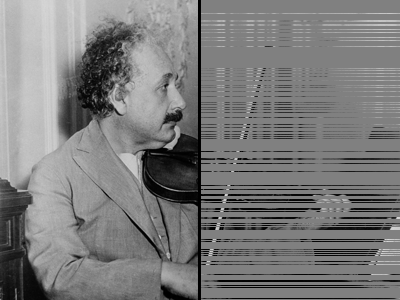}}\hfill
  \subfigure[TSC, $\alpha=10$]{\includegraphics[width=0.325\textwidth]{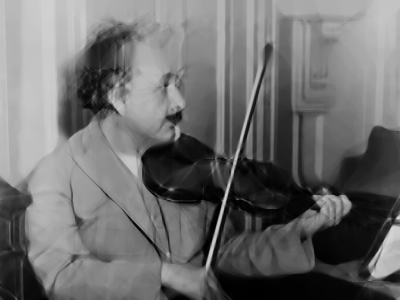}}\hfill
  \subfigure[TV, $\alpha=0$]{\includegraphics[width=0.325\textwidth]{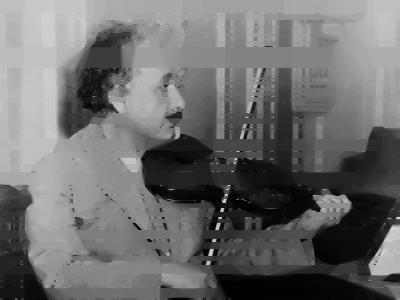}}\\
  \subfigure[Picasso]{\includegraphics[width=0.325\textwidth]{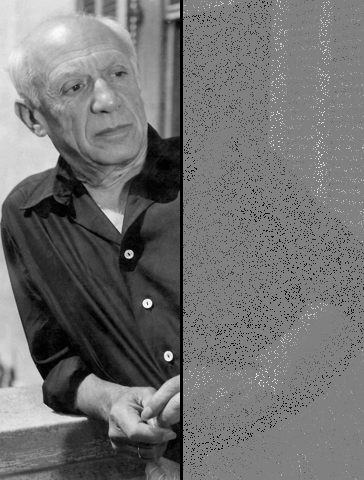}}\hfill
  \subfigure[TSC, $\alpha=5$]{\includegraphics[width=0.325\textwidth]{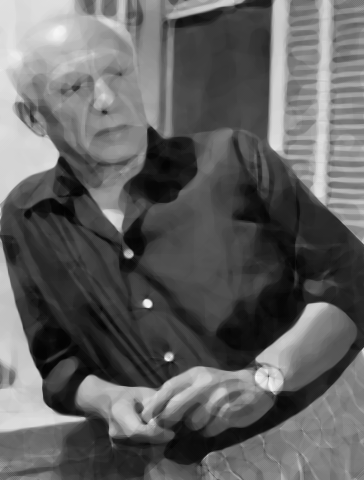}}\hfill
  \subfigure[TV, $\alpha=0$]{\includegraphics[width=0.325\textwidth]{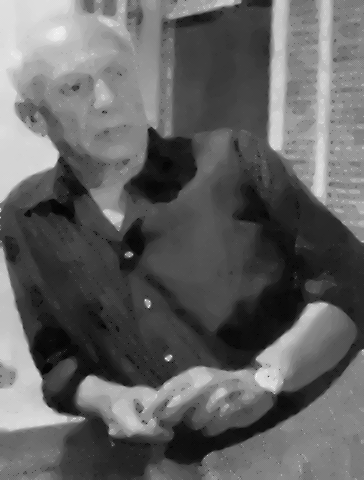}}\hfill
  \caption{Image inpainting using the proposed curvature-based
    energies. (a) shows a blend of the original ``Einstein'' image
    with its degraded version where we have removed $80\%$ of the
    lines. (d) shows a blend of the ``Picasso'' with the version where
    we have removed $90\%$ of the pixels. (b) is the inpainting result
    of the ``Einstein'' and (e) is the inpainting result of the
    ``Picasso'' image. For comparison, (c) and (f) show the result of
    standard TV regularization. Observe that our proposed curvature
    based regularization leads to significant better inpainting
    results.}\label{fig:inpainting-einstein-pablo}
\end{figure}

In this section, we apply our proposed curvature energies to the
classical problem of image inpainting. Similar to the previous
examples, we use the data term~\eqref{eq:data-term-inpainting} with
the only difference that the input image $u^0$ is now a gray level
image. In both examples we used $N_\theta=32$ discrete orientations
and the curvature parameter was set to
$\alpha=15$. Figure~\ref{fig:inpainting-einstein-pablo} shows the
results for two different inpainting problems. In the ``Einstein''
image, we randomly remove $80\%$ lines and in the ``Picasso'' image we
randomly remove $90\%$ of the pixels. In case of the ``Einstein''
image, we use TSC regularization with $\alpha=10$ and for the
``Picasso'' we set $\alpha=5$. In case of the ``Einstein'' image, the
gaps a much larger and hence we use a larger value of $\alpha$. For
comparison, we also provide results of standard TV regularization,
which is equivalent to using $\alpha=0$ in one of the three curvature
energies. From the results one can clearly see that curvature
regularization leads to significantly better inpainting results. On
the downside, we can also observe some  artifacts which are
caused by our convex representation in the roto-translation space.

\subsection{Image denoising}

\begin{figure}[t!]
  \centering
  \subfigure[Louvre]{\includegraphics[width=0.325\textwidth]{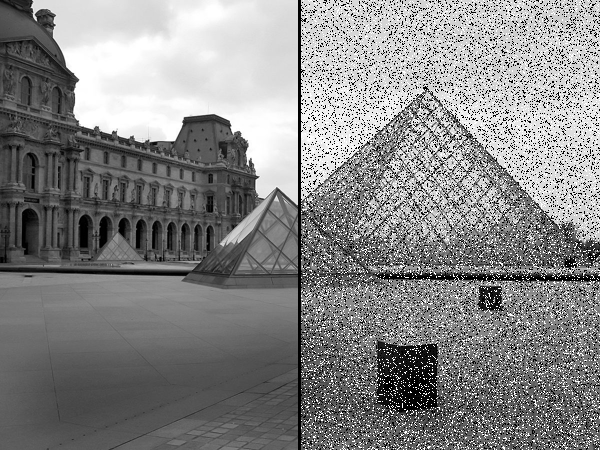}}\hfill
  \subfigure[TSC, $\alpha=10$, $\lambda=7$]{\includegraphics[width=0.325\textwidth]{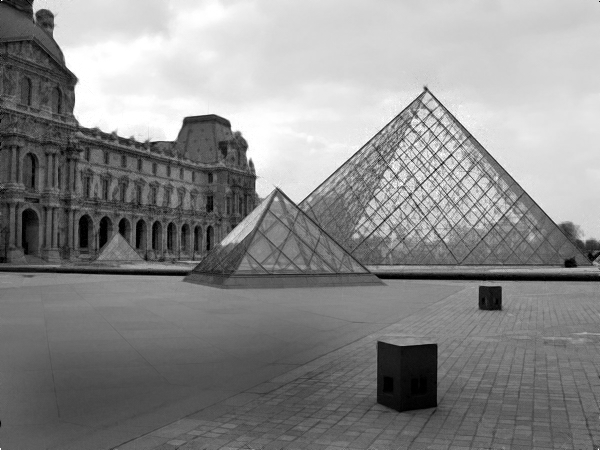}}\hfill
  \subfigure[TV, $\alpha = 0$, $\lambda=2$]{\includegraphics[width=0.325\textwidth]{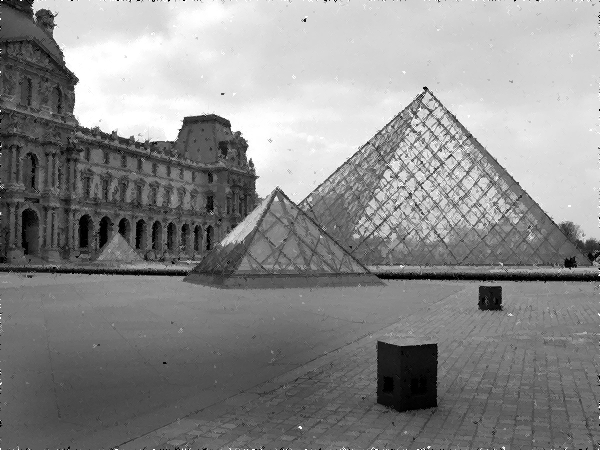}}\\
  \subfigure[Leberbl\"umchen]{\includegraphics[width=0.325\textwidth]{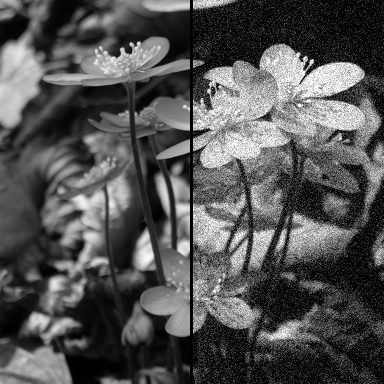}}\hfill
  \subfigure[TSC, $\alpha=10$, $\lambda=40$]{\includegraphics[width=0.325\textwidth]{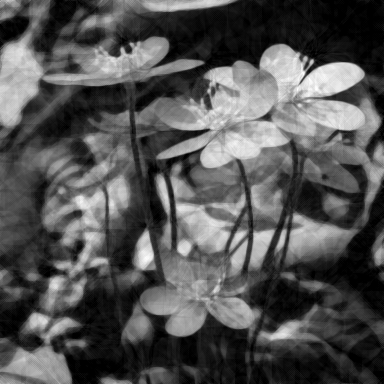}}\hfill
  \subfigure[TV, $\alpha = 0$, $\lambda=10$]{\includegraphics[width=0.325\textwidth]{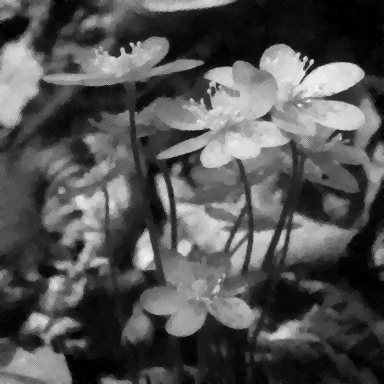}}\\
  \caption{Image denoising using curvature-based regularization. (a)
    shows the original ``Louvre'' image blended with a noisy version,
    where we have added $25\%$ ``salt \& pepper'' noise. (d) shows the
    clean ``Leberbl\"umchen'' image together with its noisy variant
    where we have added zero-mean Gaussian noise with standard
    deviation $0.1$ (b) and (e) show the result of image denoising
    using TSC regularization. For comparison we show in (c) and (f)
    the result when using classical total variation (TV)
    regularization.  }\label{fig:denoising-louvre-leberbluemchen}
\end{figure}

In our last experiments we investigate our curvature energies for the
classical problem of image denoising. We investigate two different
types of noise: Zero-mean Gaussian noise and impulse noise such as
``salt \& pepper'' noise. For Gaussian noise, it is well-known data a
quadratic $\ell_2$ data term
\[
G(u) = \frac{\lambda}{2} \norm{u-f}^2,
\]
where $\lambda > 0$ is a good data fidelity parameter. The proximal map is
given by
\[
u = \text{prox}_{\tau G}(v) \; \Longleftrightarrow \; u_\i = \frac{v_i
  + \tau\lambda f_\i}{1+\tau\lambda}, \; \forall \i \in \I.
\]
In case of impulse noise, a $\ell_1$ data term is more suitable:
\[
G(u) = \lambda \norm[1]{u-f},
\]
since it is more robust with respect to outliers. The proximal map for
the $\ell_1$ data term is given by the classical soft-shrinkage
formula
\[
u = \text{prox}_{\tau G}(v) \; \Longleftrightarrow \; u_\i = f_\i +
\max(0, |v_\i-f_\i|-\tau\lambda)\cdot\text{sgn}(v_\i-f_\i), \; \forall \i \in \I.
\]
In Figure~\ref{fig:denoising-louvre-leberbluemchen} we show the
results of our curvature-based regularization energies for image
denoising. In case of the ``Louvre'' image we generated the noisy
image by adding $25\%$ ``salt \& pepper'' noise. The noisy
``Leberbl\"umchen'' image was generated by adding zero-mean Gaussian
noise with standard deviation $0.1$. In case of Gaussian noise we use
a $\ell_2$ data term and in case of ``salt \& pepper'' noise we used
the $\ell_1$ fidelity term. In both cases, we use TSC regularization
and the curvature parameter was set to $\alpha=10$. For comparison, we
also show the results of standard total variation (TV) denoising,
which is obtained from our models by setting $\alpha=0$. One can
clearly see that TSC regularization leads to a better preservation of
image edges than TV regularization, in particular at small and
elongated structures. This effect is more visible in case of ``salt \&
pepper'' noise since the $\ell_1$ data term leads to a behavior
similar to image inpainting, once an outlier is detected.

\section{Proof of Theorem~\ref{th:C2}}\label{sec:proof}

This section and the following are devoted to the proof of
Theorem~\ref{th:C2}. We first show a preliminary result which shows
that admissible curves (of finite energy) in $\OS$ project onto curves
in $\Om$ with bounded energy as well.

\subsection{Control of curves}

Let $\gamma:[0,L]\to \OS$ a rectifiable curve of length $L$ in
$\OS$, parameterized by its length ($|\dot{\gamma}|=1$ a.e.).
We define its energy as
\[
\E(\gamma):=\int_0^1 h(\gamma(s),\dot{\gamma(s)}) ds
\]
and assume that it is finite, which in particular yields
that for a.e.~$s$, 
\[\dot{\gamma}^x(s)=|\dot{\gamma}^x(s)|
(\cos\gamma^\theta(s),\sin \gamma^\theta(s))^T.\]

We consider the projected curve $\gamma^x:[0,L]\to \Om$
(possibly overlapping even if $\gamma$ is simple)
and wish to show that the energy $\E$ controls an energy on
this curve. First we reparameterize (as usual) the rectifiable
curve $\gamma^x$ by its length.
For this we introduce the length, for $s\in [0,L]$
\[
\ell(s) := \int_0^s |\dot{\gamma}^x(t)|dt \in [0,L],
\]
which is a $1$-Lipschitz, nondecreasing function.
Since clearly for $s'>s$ one has $\ell(s)=\ell(s')$ if and only
if $\dot{\gamma}^x(t)=0$ a.e.~in $[s,s']$, we also find that
$\gamma^x(t)=\gamma^x(s)$ for all $t\in[s,s']$. This means that
(obviously) one can reparameterize $\gamma^x$ by defining a
curve $\tilde{\gamma}:[0,\ell(L)]\to \Omega$ such that
$\tilde{\gamma}(\ell(s))= \gamma^x(s)$ for all~$s\in[0,L]$.
By definition it is clear that $\tilde{\gamma}$ is $1$-Lipschitz,
moreover if $l=\ell(s)$ and $l'=\ell(s')>l$,
\[
\frac{\tilde{\gamma}(l')-\tilde{\gamma}(l)}{l'-l}
=\frac{ \int_s^{s'}\dot{\gamma}^x(t)dt}{ \int_s^{s'}|\dot{\gamma}^x(t)|dt}
\stackrel{l'\to l}\to \frac{\dot{\gamma}^x(s)}{|\dot{\gamma}^x|(s)}=\gamma^\theta(s)
\]
if $s$ is a Lebesgue point of $\dot{\gamma}^x$ where $\dot{\gamma}^x(s)\neq 0$.
One deduces that for a.e.~$l\in [0,\ell(L)]$, $\tilde{\gamma}$ has
the tangent vector
$\tau(l)=(\cos\theta(l),\sin\theta(l))^T$ given by
\[
\tau(l)=\dot{\tilde{\gamma}}(l) = 
\begin{pmatrix} \cos \gamma^\theta(s) \\
\cos \gamma^\theta(s)
\end{pmatrix}
\]
where $l=\ell(s)$.
Observe that for any smooth $\psi:(0,\ell(L))\to\R$
with compact support one has
\begin{multline*}
\int_0^{\ell(L)} \dot{\psi}(l)\theta(l)dl
= \int_0^L \dot{\psi}(\ell(s))\theta(\ell(s)) \dot{\ell}(s)ds
\\
= \int_0^L \gamma^\theta(s) (\psi\circ \ell)'(s) ds 
= -\int_0^L \dot{\gamma}^\theta(s) \psi(\ell(s)) ds\le L\|\psi\|_{C^0}
\end{multline*}
which shows that $\theta(l)$ is a $BV$ function, or equivalently that
$\tilde{\gamma}$ has a curvature $\kappa=\dot{\theta}$
which is a bounded measure; in particular
taking the supremum over $\psi$ with $|\psi|\le 1$ one finds
\[
\int_0^{\ell(L)}d|\kappa|(l) \le \int_0^L |\dot{\gamma}^\theta(s)|ds
\le L.
\]
In addition, it follows, given $\psi\in C_c^\infty(\ell(a),\ell(b))$ with
$0\le a< b \le L$ that:
\begin{multline*}
\int_{\ell(a)}^{\ell(b)} \big(-\dot{\psi}(l)\theta(l)-f^*(\psi(l))\big)dl
=
 \int_a^b \big(-(\psi\circ\ell)'(s)\gamma^\theta(s)  - f^*(\psi(\ell(s)))|\dot{\gamma}^x(s)|\big)ds
\\=
 \int_a^b \big(\psi(\ell(s))\dot{\gamma}^\theta(s)  - f^*(\psi(\ell(s)))|\dot{\gamma}^x(s)|\big)ds
\le \int_a^b h(\gamma(s),\dot{\gamma}(s))ds
\end{multline*}
%for any $\psi\in C^0(0,\ell(L))$, 
and taking the supremum one deduces that
\begin{equation*}%\label{eq:cont1}
\int_{\ell(a)}^{\ell(b)} f(\kappa) \le
\int_a^b  h(\gamma(s),\dot{\gamma}(s))ds
\end{equation*}
(where the left-hand side integral in on the open interval and denotes
a convex function of a measure).
% In addition, it follows, if $\vp\in C^0(\Om;\R_+)$ is a bounded, continuous
% and nonnegative function, that
% \begin{multline*}
% \int_0^{\ell(L)} \vp(\tilde{\gamma}(l))\big(\dot{\psi}(l)\theta(l)-f^*(\psi(l))\big)dl
% \\=
%  -\int_0^L \vp(\gamma^x(s))\big(\dot{\gamma}^\theta(s) \psi(\ell(s)) - f^*(\psi(\ell(s)))|\dot{\gamma}^x(s)|\big)ds
% \\\le \int_0^L \vp(\gamma^x(s))h(\gamma(s),\dot{\gamma}(s))ds
% \end{multline*}
% for any $\psi\in C^0(0,\ell(L))$, and taking the supremum
% one deduces that
 Now, if $\vp\in C^0(\Om;\R_+)$ is a bounded, continuous
 and nonnegative function, one deduces that
\begin{multline*}
\int_0^{\ell(L)} \vp(\tilde{\gamma}(l))f(\kappa) 
= \int_0^\infty \left(\int_{\{l:\vp(\tilde{\gamma}(l))>t\}} f(\kappa) \right)dt
\\
\le \int_0^\infty \left(\int_{\{s:\vp(\gamma^x(s))>t\}} h(\gamma(s),\dot{\gamma}(s)) \right)dt
\end{multline*}
and it follows
\begin{equation}\label{eq:cont1}
\int_0^{\ell(L)} \vp(\tilde{\gamma}(l))f(\kappa) \le
\int_0^L \vp(\gamma^x(s)) h(\gamma(s),\dot{\gamma}(s))ds.
\end{equation}

We now show the following lemma:
\begin{lemma}\label{lem:important}
Let $\Gamma\subset \Om$ be a $C^2$ (oriented) curve with
tangent $\tau_\Gamma$ and curvature $\kappa_\Gamma$, ane let
and $\gamma:[0,L]\to \OS$ be a rectifiable curve.
Define
$\Gamma^+=\{x\in\Gamma: \exists s\in(0,L), \gamma^x(s)=x \textup{ and }
 \dot{\gamma}^x(s)\cdot\tau_\Gamma\ge 0\}$.
 Then for any bounded, nonnegative, continuous function $\vp\in C^0(\Om;\R_+)$,
\begin{equation}\label{eq:boundimportant}
\int_{\Gamma^+} \vp(x) f(\kappa_\Gamma) d\H^1 \le 
\int_0^L \vp(\gamma^x(s)) h(\gamma(s),\dot{\gamma}(s))ds.
\end{equation}
\end{lemma}
\begin{proof}
We just need to show that, defining the measure $\kappa$ as before,
\[
\int_{\Gamma^+} \vp(x)f(\kappa_\Gamma) d\H^1 \le 
\int_0^{\ell(L)} \vp(\tilde{\gamma})f(\kappa),
\]
and the conclusion will follow from~\eqref{eq:cont1}.

A first observation is that as $\tau(l)$ defined above is $BV$,
it has at most a countable number of jumps and
one can cover (up to the jump points) $[0,\ell(L)]$
with an at most countable
union of intervals $(a_i,b_i)_{i\in I}$  on which $\tau(l)$ is continuous
(hence $\tilde{\gamma}$ is a $C^1$ curve on $(a_i,b_i)$). Moreover,
possibly dropping further a finite number of points, one may assume
that $|\kappa|(a_i,b_i)<\pi$ so that $\gamma^x(a_i,b_i)$ is
a simple curve.

Assume $|\{l\in (a_i,b_i)\,:\, \tilde{\gamma}(l)\in\Gamma,
\dot{\tilde{\gamma}}(l)\cdot\tau_\Gamma\ge 0\}|>0$. Choose such
an $l\in(a_i,b_i)$, with $x=\tilde{\gamma}(l)\in\Gamma^+$, and such
that $\tilde{\gamma}$ is differentiable in $l$.
We assume in addition that $x$ is a point of ($\H^1$-)density
one in $\tilde{\gamma}(a_i,b_i)\cap \Gamma^+$,
that $l$ is a Lebesgue point of $\kappa^a$,
the absolutely continuous part (w.r.~the Lebesgue measure) of
$\kappa$, and that $\theta(l)$ is approximately differentiable
at $l$: for any $\eta>0$,
\[
\lim_{\e\to 0}\frac{1}{2\e}\left|
\left\{ s\in (l-\e,l+\e):
\frac{|\theta(s)-\theta(l)-\kappa^a(l)(s-l)|}{|s-l|}ds > \eta \right\}\right|
=0.
\]
% \[
% \lim_{\e\to 0} \frac{1}{2\e}\int_{l-\e}^{l+\e} \frac{|\theta(s)-\theta(l)-\kappa^a(l)(s-l)|}{|s-l|}ds = 0. %% check this!! [a bit sloppy as $\theta\in\Sp^1$]
% \]
All this is true $\H^1$-a.e.~in $\tilde{\gamma}(a_i,b_i)\cap \Gamma^+$,
so there is no loss of generality, see for instance~\cite{AmbrosioFuscoPallara}.

If $\dot{\tilde{\gamma}}(l)\neq \tau_{\Gamma}(x)$,
then $x$ must be an isolated point
in $\tilde{\gamma}(a_i,b_i)\cap\Gamma^+$, a contradiction.
Hence $\dot{\tilde{\gamma}}(l)= \tau_{\Gamma}(x)$. In the same
way, one has that for a.e.~$s$ near $l$ such that
$\tilde{\gamma}(s)\in\Gamma$, $\dot{\tilde{\gamma}}(s)=
 \tau_{\Gamma}(\tilde{\gamma}(s))$. For such $s$,
one therefore has that
\[
\frac{|\theta(s)-\theta(l)-\kappa^a(l)(s-l)|}{|s-l|}
= 
\frac{|\theta_\Gamma(\tilde{\gamma}(s))-\theta_\Gamma(x)-\kappa^a(l)(s-l)|}{|s-l|},
\]
where $\theta_\Gamma$ is the angle between
{\scriptsize$\begin{pmatrix}1\\0\end{pmatrix}$
and $\tau_\Gamma$}.
Using that for such $s$,
\begin{multline*}
\theta_\Gamma(\tilde{\gamma}(s))-\theta_\Gamma(x)
= \kappa_\Gamma(x)(\tilde{\gamma}(s)-x)\cdot\tau_\Gamma(x)
+ o(|\tilde{\gamma}(s)-x|) 
\\ = 
\kappa_\Gamma(x)\dot{\tilde{\gamma}}(l)\cdot\tau_\Gamma(x)(s-l)
+ o(|\dot{\tilde{\gamma}}(s)||s-l|)
\end{multline*}
we obtain that for such $s$,
\[
\frac{|\theta(s)-\theta(l)-\kappa^a(s-l)|}{|s-l|}
= |\kappa_\Gamma(x)-\kappa^a(l)| + o(1)
\]
(where $o(1)$ goes to $0$ as $\e\to 0$).
Hence, for $\e>0$ small and $\eta>0$,
\begin{multline*}
\left\{
s\in (l-\e,l+\e):
\frac{|\theta(s)-\theta(l)-\kappa^a(l)(s-l)|}{|s-l|}ds > \eta
\right\}
\\ \supseteq 
\left\{
s\in (l-\e,l+\e): \tilde{\gamma}(s)\in\Gamma
\textup{ and }|\kappa_\Gamma(x)-\kappa^a(l)| + o(1) > \eta
\right\}
\end{multline*}
% \begin{multline*}
% \frac{1}{2\e}\int_{l-\e}^{l+\e} \frac{|\theta(s)-\theta(l)-\kappa^a(l)(s-l)|}{|s-l|}ds
% \\
% \ge \frac{|\{s\in (l-\e,l+\e): \tilde{\gamma}(s)\in\Gamma\}|}{2\e}( |\kappa_\Gamma(x)-\kappa^a(l)| + o(1)).
% \end{multline*}
In the limit, using that $l$ is a point of density one in
the set $\{\tilde{\gamma}(\cdot)\in\Gamma\}$, we obtain that
$|\kappa_\Gamma(x)-\kappa^a(l)|\le \eta$ and since $\eta$ is arbitrary,
$\kappa_\Gamma(x)=\kappa^a(l)$.
Since the intervals $(a_i,b_i)$ cover $\H^1$-almost all~of $[0,\ell(L)]$,
we find that this equality holds $\H^1$-a.e.~in $\Gamma^+$.

The thesis of the Lemma easily follows, since
\[
\int_0^{\ell(L)}  \vp(\tilde{\gamma})f(\kappa)\ge 
\int_0^{\ell(L)}  \vp(\tilde{\gamma})f(\kappa^a)dx
%\ge \sum_{i\in I} \int_{a_i}^{b_i}  \vp(\tilde{\gamma}) f(\kappa^a)dx.
\]
\end{proof}
\subsection{Decomposition of $\sigma$}

Let $u\in BV(\Om)$ with $F(u)<\infty$.
We observe that we can decompose an admissible $\sigma$
(with in particular $\Div\sigma=0$) as follows:
\begin{equation}\label{eq:totaldecomposition}
\sigma = \int_{\M} \lambda d\mu(\lambda),\quad
|\sigma| = \int_{\M} |\lambda| d\mu(\lambda),\quad
\end{equation}
where $\mu$ is a nonnegative measure on the set of charges $\M$
which is here a shorthand notation for $\M(\OS;\R^3)$:
%with zero divergence
see Appendix~\ref{app:Smirnov} which discusses the results of
Smirnov in~\cite{Smirnov}.

\begin{lemma}\label{lem:totdecg}
 Assume~\eqref{eq:totaldecomposition} holds.
Then for any $g(\xi,p)$ nonnegative, continuous, convex and one-homogeneous
in its second argument,
\begin{equation}\label{eq:totdecg}
\int_{\OS} g(\xi,\sigma) = \int_{\M}\left(\int_{\OS} g(\xi,\lambda)\right)d\mu(\lambda).
\end{equation}
\end{lemma}

\begin{proof}
For $\rho>0$ we define $\Om_\rho=\{x\in \Om:\dist(x,\partial\Om)>\rho\}$.
We let for $\rho>0$, $\e\in (0,\rho)$ and
$\xi\in \OS[2\rho]$
\[
f_{\rho,\e}(\xi)=
\begin{cases}
 \frac{\sigma(\OS[3\rho]\cap\ov{B}(\xi,\rho))}{|\sigma|(B(\xi,\rho+\e))}
&  \textup{ if }
|\sigma|(B(\xi,\rho+\e))>0\,,\\
0 & \textup{ else.}
\end{cases}
\]
We also let $f_{\rho,\e}(\xi)=0$ if $\xi\in\OS\setminus \OS[2\rho]$:
remark then that $f_{\rho,\e}\in C_c^0(\OS;B(0,1))$ and that
\[
f_\rho(\xi):=\lim_{\e\to 0} f_{\rho,\e}(\xi)= \begin{cases}
\frac{\sigma(\OS[3\rho]\cap\ov{B}(\xi,\rho))}{|\sigma|(\ov{B}(\xi,\rho))} & \textup{ if } \xi\in\Om_{2\rho}\times\Sp^1,\,
|\sigma|(\ov{B}(\xi,\rho))>0\,,\\
0 & \textup{ else.}
\end{cases}
\]

By~\eqref{eq:totdecg} we have for any $\vp\in C_c^0(\OS;\R_+)$:
\[
\int_{\OS} \vp(|\sigma|-f_{\rho,\e}\cdot \sigma)
= \int_{\M}\left(\int_{\OS}
\vp(|\lambda|-f_{\rho,\e}\cdot \lambda)\right)d\mu.
\]
Observe that
\begin{multline*}
\int_{\OS}\vp(|\lambda|-f_{\rho,\e}\cdot \lambda)
\\= \int_{\OS}\vp\left(|\lambda|-f_{\rho}\cdot \lambda\right)\frac{{|\sigma|(\ov{B}(\xi,\rho))}}{{|\sigma|(B(\xi,\rho+\e))}}
+ \int_{\OS} \vp\frac{|\sigma|(B(\xi,\rho+\e)\setminus \ov{B}(\xi,\rho))}{|\sigma|(B(\xi,\rho+\e))}|\lambda|
\end{multline*}
and using the monotone convergence theorem, we can send $\e\to 0$ and
deduce that
\begin{equation}\label{eq:RN0}
\int_{\OS} \vp(|\sigma|-f_{\rho}\cdot \sigma)
= \int_{\M}\left(\int_{\OS}
\vp(|\lambda|-f_{\rho}\cdot \lambda)\right)d\mu.
\end{equation}
Then, thanks to Radon-Nikodym's derivation theorem (and
Lebesgue's convergence theorem),
\[
\lim_{\rho\to 0} \int_{\OS} \vp(|\sigma|-f_{\rho}\cdot \sigma)=0.
\]
Observe that all the integrands in~\eqref{eq:RN0} are nonnegative.
Let $\vp_n$ be a nondecreasing sequence of compactly supported
nonnegative smooth functions which converges to $1$. Then
one can build a subsequence $\rho_n\downarrow 0$ such that
\[
\sum_n \int_{\OS} \vp_n(|\sigma|-f_{\rho_n}\cdot \sigma) < +\infty,
\]
so that
\[
\int_{\M} \left(\int_{\OS} \sum_n  \vp_n(|\lambda|-f_{\rho_n}\cdot \lambda)
\right)d\mu < +\infty.
\]
It follows that there exists
a set $E\subset\M$ with $\mu(E)=0$ such that if $\lambda\not\in E$,
\[
\lim_n \int_{\OS} \vp_n(|\lambda|-f_{\rho_n}\cdot \lambda) = 0
\]
and $f_{\rho_n}\cdot \frac{\lambda}{|\lambda|}\to 1$ $|\lambda|$-\ale~in~$\OS$
(using $|\lambda|-f_{\rho_n}\cdot \lambda= |\lambda|(1-f_{\rho_n}\cdot \frac{\lambda}{|\lambda|})\ge 0$).
Using $|f_{\rho_n}-\frac{\lambda}{|\lambda|}|^2= |f_{\rho_n}^2|+1-2f_{\rho_n}\cdot \frac{\lambda}{|\lambda|}\le 2(1-f_{\rho_n}\cdot \frac{\lambda}{|\lambda|})$
it follows
that $f_{\rho_n}\to \frac{\lambda}{|\lambda|}$, $|\lambda|$-\ale

Let $g:(\OS)\times \R^3\to \R$ be a continuous, nonegative,
and convex one-homogeneous function in its second argument.
Then as before for $\vp\in C_c^0(\OS;\R_+)$
\[
\int_{\OS} \vp(\xi) g(\xi,f_{\rho,\e}(\xi))|\sigma|=\int_{\M}\int_{\OS}  \vp(\xi) g(\xi,f_{\rho,\e})|\lambda| d\mu
\]
and since $g(\xi,f_{\rho,\e}(\xi))=(|\sigma|(\ov{B}(\xi,\rho))/|\sigma|(B(\xi,\rho+\e))) g(\xi,f_\rho(\xi))\uparrow g(\xi,f_\rho(\xi))$ as $\e\downarrow 0$,
the monotone convergence theorem yields that
\[
\int_{\OS} \vp g(\xi,f_{\rho})|\sigma|=\int_{\M}\int_{\OS} \vp g(\xi,f_{\rho})|\lambda| d\mu.
\]
Using Lebesgue's theorem, we have
\[
\lim_{\rho\to 0} \int_{\OS} \vp(\xi) g(\xi,f_\rho(\xi))|\sigma|=
\int_{\OS} \vp(\xi) g(\xi,\tfrac{\sigma}{|\sigma|}(\xi))|\sigma|=
 \int_{\OS} \vp(\xi) g(\xi,\sigma).
\]
On the other hand, we know that if $\lambda\not\in E$, still thanks
to Lebesgue's theorem,
\[
\lim_{n} \int_{\OS} \vp(\xi) g(\xi,f_{\rho_n}(\xi))|\lambda| = 
\int_{\OS} \vp(\xi) g(\xi,\tfrac{\lambda}{|\lambda|}(\xi))|\lambda| = \int_{\OS} \vp g(\xi,\lambda).
\]
We invoke one last time Lebesgue's theorem to deduce that
(for all  nonnegative test function $\vp$)
\[
\lim_{n} \int_{\M} \left(\int_{\OS} \vp(\xi) g(\xi,f_{\rho_n}(\xi))|\lambda|\right)d\mu = 
\int_{\M} \left( \int_{\OS} \vp g(\xi,\lambda) \right)d\mu.
\]
We deduce~\eqref{eq:totdecg}.
\end{proof}

\begin{corollary} Let $\sigma$ be admissible for problem~\eqref{eq:deFu},
and assume it is decomposed as in~\eqref{eq:totaldecomposition}.
Then $\mu$-\ale~measure $\lambda$ satisfies
\[
\frac{\lambda}{|\lambda|}\cdot\theta \ge 0\,,
\frac{\lambda}{|\lambda|}\cdot\theta^\perp = 0\, \ale~\textup{in }\OS
\]
\end{corollary}

\begin{proof}
Apply the Lemma with $g(\xi,p) = (\theta\cdot p)^-$ and
then with $g(\xi,p)= |\theta^\perp\cdot p|$.
\end{proof}
\begin{corollary} Let $\sigma$ be admissible for problem~\eqref{eq:deFu},
and assume it is decomposed as in~\eqref{eq:totaldecomposition}.
Let $h$ be defined as in~\eqref{eq:defh}, and $\vp\in C_c^0(\Om)$
a nonnegative test function. Then
\begin{equation}\label{eq:totdech}
\int_{\OS} \vp(x)h(\theta,\sigma) = \int_{\M}\left(\int_{\OS} \vp(x)h(\theta,\lambda)\right)d\mu(\lambda).
\end{equation}
\end{corollary}
\begin{proof}
This is a consequence of Lemma~\ref{lem:totdecg}, however now $h$ can take the
value $+\infty$. We simply observe that if we let, for $\e>0$,
\[
h_\e(\theta,p):= \min_{q\in\R^3} h(\theta,q)+\frac{1}{\e}|p-q|
\]
then $h_\e$ is $(1/\e)$-Lipschitz, convex, one-homogeneous in the
second variable, and continuous (and hence also $(\xi,p)=((x,\theta),p)
\mapsto\vp(x)h_\e(\theta,p)$).
Hence~\eqref{eq:totdecg} yields
\begin{equation*}%\label{eq:totdech}
\int_{\OS} \vp(x)h_\e(\theta,\sigma) = \int_{\M}\left(\int_{\OS}\vp(x) h_\e(\theta,\lambda)\right)d\mu(\lambda).
\end{equation*}
The results follows from the monotone convergence theorem.
\end{proof}

\subsection{Conclusion: proof of Theorem~\ref{th:C2}}

We can now complete the proof of Theorem~\ref{th:C2}. We consider $E$ a $C^2$
set. First, if $\gamma(t):[0,1]\to\Om$ is an
oriented parameterization of (the closure of) $\partial E\cap \Omega$,
(assuming it is connected, otherwise one needs to introduce a curve
for each connected component), one can define in $\OS$
the curve $\lambda(t)=(\gamma(t),\theta(t))$ where $\theta(t)$ is
defined by $\lambda(t)=|\gamma'(t)|(\cos \theta(t),\sin\theta(t))^T$.
We then let $\Gamma=\lambda([0,1])\cap \OS$ and
$\sigma=\tau_{\Gamma}\H^1\restr\Gamma$. Then, $\sigma$ is admissible
for~\eqref{eq:deFu} with $u=\chi_E$, and one finds that
\[
F(\chi_E)\le \int_{\OS}  h(\theta,\sigma) = \int_{\partial E\cap \Om}f(\kappa_E)d\H^1.
\]

We need therefore to prove the reverse inequality. Let $\sigma$
be admissible for~\eqref{eq:deFu}. Thanks to~\cite[Theorem~A]{Smirnov}
(\textit{cf}~Appendix~\ref{app:Smirnov}, eq.~\eqref{eq:SmirnovCurves}),
one can decompose $\sigma$ as
\[
\sigma = \int_{\Cone}\lambda d\mu(\lambda),\quad 
|\sigma| = \int_{\Cone}|\lambda| d\mu(\lambda),
\]
where $\lambda$ are of the form
\[
\lambda_\gamma = \tau_{\gamma} \H^1\restr\gamma
\]
for rectifiable (possibly closed) curves $\gamma\subset\OS$ of length
at most one.

Equation~\eqref{eq:totdech} is valid for this decomposition and
shows that for any $\vp\in C_c^0(\Om;[0,1])$,
\begin{equation}\label{eq:totdechgamma}
\int_{\OS} \vp(x)h(\theta,\sigma) = \int_{\Cone}\left(\int_{\gamma} \vp(x)h(\theta,\tau_\gamma)d\H^1\right)d\mu(\lambda_\gamma).
\end{equation}

Now, for any $\psi\in C_c^0(\Om;\R^2)$, one has
\begin{equation}\label{eq:aaa}
\int_{\OS} (\psi,0)^T\cdot\sigma = \int_\Om \psi\cdot Du^\perp
= \int_{\partial \Om} \psi\cdot \tau_E d\H^1
\end{equation}
where $\tau_E$ is a tangent vector to $\partial E$
(oriented with $E$ on the left-hand side and $E^c$ on the right-hand
side\footnote{This just depends on the choice of the $90^\circ$ rotation
$x\mapsto x^\perp$.}).

Let us introduce the signed distance function $d_E(x)=\dist(x,\Om\setminus E)
-\dist(x,E)$ which is $C^2$ in a neighborhood of $\partial E$
and is such that $\nabla d_E=\nu_E$ (the inner normal)
on $\partial\Om$ and $\Delta d_E=-\kappa_E$ (we assume the
curvature is nonnegative where the set is convex).
If we consider, for $\e>0$, a test function 
of the form
\[
\psi(x)=(\nabla d_E(x))^\perp \left(1-\frac{|d_E(x)|}{\e}\right)^+\vp(x),
\]
with $\vp\in C_c^0(\Om)$, \eqref{eq:aaa} yields
\begin{multline*}
\int_{\partial E}\vp(x)d\H^1 = \int_{\OS}(\psi,0)^T\cdot\sigma
\\=\int_{\Cone} 
\left(\int_\gamma \vp(x)\left(1-\frac{|d_E(x)|}{\e}\right)^+
((\nabla d_E(x))^\perp,0)^T\cdot\tau_{\gamma}(x,\theta) d\H^1\right)d\mu(\lambda_\gamma).
\end{multline*}
Sending $\e\to 0$ %% Monotone (splitting in positive/negative parts)? Lebesgue?
we find that
\begin{equation*}
\int_{\partial E}\vp(x)d\H^1=\int_{\Cone} 
\left(\int_{\gamma\cap (\partial E\times\Sp^1)} \vp\, (\tau_E,0)^T\cdot\tau_{\gamma} d\H^1\right)d\mu(\lambda_\gamma).
\end{equation*}
Now, we can choose $\vp$ of the form $\vp(x)f(\kappa_E(x))$,
 with $\vp\in C_c^0(\Om;[0,1])$, assuming $f$ is finite-valued.
We obtain that
\begin{equation}\label{bbb}
\int_{\partial E}\vp f(\kappa_E)d\H^1=\int_{\Cone} 
\left(\int_{\gamma\cap (\partial E\times\Sp^1)}
 \vp f(\kappa_E)\tau_E\cdot\tau^x_{\gamma} d\H^1\right)d\mu(\lambda_\gamma).
\end{equation}
Given the rectifiable curve $\gamma$, one has that
\begin{multline}\label{ccc}
\int_{\gamma\cap(\partial E\times \Sp^1)} \vp f(\kappa_E)\tau_E\cdot\tau^x_{\gamma} d\H^1
\le \int_{\Pi(\gamma)\cap\partial E} \vp f(\kappa_E)(\tau_E\cdot\tau_{\Pi(\gamma)})^+d\H^1
\\\le \int_{\OS} \vp(x)h(\theta,\lambda_\gamma)
\end{multline}
thanks to~\eqref{eq:boundimportant} in Lemma~\ref{lem:important}.
Then, \eqref{eq:totdechgamma}, \eqref{bbb} and~\eqref{ccc} yield
\[
\int_{\partial E}\vp(x) f(\kappa_E) d\H^1 \le \int_{\OS} \vp(x)h(\theta,\sigma)
\]
and Theorem~\ref{th:C2} is easily deduced. In case $f$ takes the
value $+\infty$, then we first approximates $f$ with a finite-valued
function from below (replacing $f$ with $\min_{t'} f(t')+|t-t'|/\e$
for $\e>0$ small), and once the inequality is established for
this function we send $\e\to 0$, so that~\eqref{eq:goodrelaxE} also
holds.

\section*{Acknowledgements}

The authors would like to thank
the Isaac Newton Institute for Mathematical Sciences,
Cambridge, for support and hospitality during the programme
``Variational methods, new optimisation techniques and new fast numerical algorithms'' (Sept.-Oct., 2017), when this paper was completed.
This work was supported by: EPSRC Grant N.~EP/K032208/1.
 The work of A.C.~was also
partially supported by a grant of the Simons
Foundation. T.P.~acknowledges support by the Austrian science fund
(FWF) under the project EANOI, No. I1148 and the ERC starting grant
HOMOVIS, No. 640156.

\bibliography{RT}
%\printbibliography

\appendix

\section{Consistency of the discretization}\label{app:consist}

In this appendix, we study the consistency of the discrete
approximation of the problem
which is used in Section~\ref{sec:numexp}.

\subsection{Preliminary results}
Let us introduce, for $(s,t)\in \R^2$,
\[
\bar h(s,t) = \begin{cases} s f(t/s) & \textup{ if } s>0,\\
f^\infty(t) & \textup{ if } s=0,\\
+\infty & \textup{else,}
\end{cases}
\]
which is such that
\[
\bar h(s,t) = \sup_{a+f^*(b)\le 0} as+bt.
\]
Observe that if the convex function $f$ is differentiable,
then for $s>0$, $\partial_s \bar h(s,t)=f(t/s)-(t/s)f'(t/s)\le f(0)$.

Consider a vector-valued measurable function $\sigma(\theta)
=(\lambda(\theta)\utheta,\mu(\theta))$
where $\lambda\ge 0$. Let then $\bar\theta\in \Sp^1$, $\dt>0$ and
\[
\bar\sigma =\frac{1}{\dt}\int_{\bar\theta-\frac{\dt}{2}}^{\bar\theta+\frac{\dt}{2}} \sigma(\theta)d\theta.
\]
% We claim that if $\dt$ is small enough, %% WRONG
% \[
% \bar h(\bar\sigma\cdot\utheta_0,\bar\sigma^\theta)
% \le \left(1+\frac{f(0)}{12\gamma} \dt^2\right)
% \frac{1}{\dt}\int_{\bar\theta-\frac{\dt}{2}}^{\bar\theta+\frac{\dt}{2}} h(\theta,\sigma(\theta))d\theta.
% \]
We first observe that
\[
\bar\sigma\cdot\bar\utheta=\frac{1}{\dt}\int_{\bar\theta-\frac{\dt}{2}}^{\bar\theta+\frac{\dt}{2}}
\lambda(\theta)\utheta\cdot \bar\utheta d\theta
\in \left[\bar \lambda \cos\tfrac{\dt}{2},\bar\lambda\right]
\]
where
\[
\bar\lambda=\frac{1}{\dt}\int_{\bar\theta-\frac{\dt}{2}}^{\bar\theta+\frac{\dt}{2}} \lambda(\theta)d\theta.
\]
Assuming $f$ is smooth, it follows that
\begin{multline*}
%\bar h\left(\bar \sigma\cdot\bar\utheta,\bar\sigma^\theta\cos\frac{\dt}{2}\right)
%= \cos\frac{\dt}{2} 
\bar h \Big(\tfrac{\bar \sigma\cdot\bar\utheta}{\cos\frac{\dt}{2}},\bar\sigma^\theta\Big)
=
\bar h(\bar \lambda,\bar\sigma^\theta)
+\int_0^1 \partial_s 
\bar  h \Big(\bar\lambda+s\Big(\tfrac{\bar \sigma\cdot\bar\utheta}{\cos\frac{\dt}{2}}-\bar\lambda\Big),\bar\sigma^\theta\Big)
\Big(\tfrac{\bar \sigma\cdot\bar\utheta}{\cos\frac{\dt}{2}}-\bar\lambda\Big)
ds
\\
\le 
\bar h(\bar \lambda,\bar\sigma^\theta)
 + f(0)\bar\lambda\frac{1-\cos\frac{\dt}{2}}{\cos\frac{\dt}{2}}.
\end{multline*}
Using that $\bar h$ is $1$-homogeneous and that $\bar h(s,t)\ge\gamma\sqrt{s^2+t^2}$, it follows
\begin{multline*}
\bar h \big(\bar \sigma\cdot\bar\utheta,\cos\tfrac{\dt}{2}\bar\sigma^\theta\big)
\le \cos\tfrac{\dt}{2} \bar h(\bar \lambda,\bar\sigma^\theta)
+ \big(1-\cos\tfrac{\dt}{2}\big)
\frac{f(0)}{\gamma} \bar h(\bar \lambda,\bar\sigma^\theta)
\\\le \left(
1+\frac{\dt^2}{4}(f(0)/\gamma-1)\right)\bar h(\bar \lambda,\bar\sigma^\theta)
\end{multline*}
if $\dt$ is small enough. Observe that if $f$ is not smooth, this still
holds by approximation. It follows, thanks to Jensen's inequality
and the fact that $h(\theta,\sigma)=\bar h(\lambda,\sigma^\theta)$ for
all $\theta$, that provided $\dt$ is small enough (not depending
on anything)
\begin{equation}\label{eq:averageestimateh}
\bar h \big(\bar \sigma\cdot\bar\utheta,\cos\tfrac{\dt}{2}\bar\sigma^\theta\big)
\le
(1+C{\dt^2})\frac{1}{\dt}
\int_{\bar\theta-\frac{\dt}{2}}^{\bar\theta+\frac{\dt}{2}} h(\theta,\sigma(\theta))d\theta,
\end{equation}
where the constant $C=(f(0)/\gamma-1)/4$ only depends on the function $f$.
\begin{remark}\label{rem:smeasure}\textup{
If $\sigma$ is a bounded measure such that
$\int h(\theta,\sigma)<+\infty$, then~\eqref{eq:averageestimateh}
still holds with now all integrals and averages
replaced with integrals over $[\bar\theta-\dt/2,\bar\theta+\dt/2)$.
Indeed, in this case, approximating
$\sigma$ with smooth measures by convolution one easily deduces
that it holds for almost all $\bar\theta$
(whenever $|\sigma|(\{\bar\theta-\dt/2,\bar\theta+\dt/2\})=0$). 
Then, for the other values, it is enough to find a sequence
$\e_n>0$ with $\e_n\downarrow 0$ such that the result holds for
$\bar\theta-\e_n$ and use the fact that for any bounded measure $\mu$,
\[
\lim_{n\to\infty} \mu([\bar\theta-\e_n-\tfrac{\dt}{2},\bar\theta-\e_n+\tfrac{\dt}{2}))=
\mu([\bar\theta-\tfrac{\dt}{2},\bar\theta+\tfrac{\dt}{2}))
\]
as $\mu([\theta-\e_n,\theta))\to 0$ as $n\to\infty$ for any $\theta\in\Sp^1$.
}\end{remark}

We can now show the following lemma:
\begin{lemma}\label{lem:controlmean}
Let $\dx,\dt>0$ small enough, $\sigma\in \M^1(\OS;\R^3)$ with
$\Div\sigma=0$ and such that $\int_{\OS} h(\theta,\sigma)<\infty$. Define,
for $(\bar x,\bar\theta)\in\OS$,
\[S=\left[\bar x_1-\tfrac{\dx}{2},\bar x_1+\tfrac{\dx}{2}\right)
\times\left[\bar x_2-\tfrac{\dx}{2},\bar x_2+\tfrac{\dx}{2}\right)
\]
and
\[
\bar\sigma = \frac{1}{\dx^2\dt}\sigma(S\times [\bar\theta-\tfrac{\dt}{2},\bar\theta+\tfrac{\dt}{2})).
\]
Then, 
\begin{equation}\label{eq:controlmean}
\bar h(\bar\sigma\cdot\bar\utheta,\cos\tfrac{\dt}{2}\bar\sigma^\theta)
\le \frac{1+C\dt^2}{\dx^2\dt} \int_{S\times [\bar\theta-\frac{\dt}{2},\bar\theta+\frac{\dt}{2})} h(\theta,\sigma)
\end{equation}
where $C$ depends only on $f$. Moreover, $\bar\sigma^x/|\bar\sigma|$ lies
in the cone $\R_+(\underline{\theta-\dt/2})+\R_+(\underline{\theta+\dt/2})$.
\end{lemma}
\begin{proof}
If we introduce the (averaged) marginal $\sigma'\in \M^1(\Sp^1;\theta)$ defined
by $\int_{\Sp^1}\psi \sigma' = \frac{1}{\dx^2}\int_{S\times \Sp^1}\psi\sigma$ for all
$\psi\in C^0(\Sp^1)$, then one observes that
\[
\bar\sigma = \frac{1}{\dt}\sigma'({[\bar\theta-\tfrac{\dt}{2},\bar\theta+\tfrac{\dt}{2})})
\]
and
\[
\int_{[\bar\theta-\frac{\dt}{2},\bar\theta+\frac{\dt}{2})} h(\theta,\sigma')
\le \int_{S\times [\bar\theta-\frac{\dt}{2},\bar\theta+\frac{\dt}{2})} h(\theta,\sigma).
\]
This follows by a disintegration argument and using Jensen's inequality in
each ``slice'' corresponding to a fixed value of $\theta$. The result
then follows from~\eqref{eq:averageestimateh}, together with Remark~\ref{rem:smeasure}.

The last statement comes from the fact that $\sigma_x$ is the average
of measures all contained in the cone, which is convex.
\end{proof}

Consider now a measure $\sigma$ admissible for some function $u$,
and assume that $\sigma$ is ``smooth'' in $x$: 
we assume for instance that it is the result of a convolution
$\rho_\e*\sigma'$ for some $\sigma'$ admissible (possibly extended
in a larger domain), with $\rho_\e(x)$ a rotationally symmetric
mollifier, as in the proof of Proposition~\ref{thm:smoothapprox}.
In this case, $x\mapsto \sigma$ can be seen as
a $\M^1(\Sp^1;\R^3)$-valued smooth function.

Consider $(\bar x,\bar\theta)\in\OS$, $\dx,\dt$ small, and define 
in the volume $V=S\times [\bar\theta-\dt/2,\bar\theta +\dt/2)$
the average $\bar\sigma$ as before and the ``Raviart-Thomas'' approximation
of $\sigma$ defined by the average fluxes through the 6 facets of $V$
(linearly extended inside the volume, as in eq.~\ref{eq:RaTho}).
There are several ways to define this properly,
at least for all $\bar\theta$ but a countable number.
In our case, one
can disintegrate the measure in $\OS$ as $\sigma=\sigma_\theta d\mu$
where $\mu$ is a bounded positive measure in $\Sp^1$ and 
for all $\theta$, $\sigma_\theta\in C^\infty(\Om)$. Then
\[
\tau^1_\pm = \frac{1}{\dx\dt}\int_{[\bar\theta-\dt/2,\bar\theta +\dt/2)}\left(
\int_{\bar x_2-\frac{\dx}{2}}^{\bar x_2+\frac{\dx}{2}}
 \sigma^1_\theta(\bar x_1\pm\tfrac{\dx}{2},x_2,\theta)dx_2 \right)\mu,
\]
\[
\tau^2_\pm = \frac{1}{\dx\dt}\int_{[\bar\theta-\dt/2,\bar\theta +\dt/2)}\left(
\int_{\bar x_1-\frac{\dx}{2}}^{\bar x_1+\frac{\dx}{2}}
 \sigma^1_\theta(x_1,\bar x_2\pm\tfrac{\dx}{2},\theta)dx_1 \right)\mu.
\]
To define the  vertical fluxes $\tau^\theta_\pm$ we 
assume in addition that $\mu(\{\bar\theta\pm\dt/2\})=0$
(which is true for all values but a countable number). In this
case, observe that if
$\phi\in C_c^1(\mathring{S}\times\{\theta=-\dt/2\})$, it can be extended
into a $C^1$ function in $V$ vanishing near the 5 other boundaries,
and then 
\[
\int_V \nabla\phi\cdot \sigma
\]
defines a measure $\tilde\tau^\theta_-$ on $\mathring{S}\times\{\theta=-\dt/2\}$.
Then one simply let $\tau^\theta_- = \tilde\tau^\theta_-(\mathring{S}\times\{\theta=-\dt/2\})/\dx^2$. 
The value $\tau^\theta_+$ is defined in the same way.
The assumption that $\mu(\{\bar\theta-\dt/2\})=0$
guarantees that the same construction from below will build the same measure
and the same value, and that one actually has
$(\tau^1_+-\tau^1_-+\tau^2_+-\tau^2_-)/\dx+(\tau^\theta_+-\tau^\theta_-)/\dt=0$.

We can show the following lemma.
\begin{lemma}
Let $\tau = (\tau^1_a,\tau^2_b,\tau^\theta_c)^T$ for any $(a,b,c)\in \{-,+\}^3$.
Then, for all~$\bar\theta$ but a countable number,
\begin{equation}\label{eq:RT0vsAve}
\dx^2\dt|\tau-\bar\sigma| \le \sqrt{\dx^2+\dt^2}\int_{V} |\partial_1\sigma^1|+|\partial_2\sigma^2|.
\end{equation}
\end{lemma}
\begin{proof}
We prove the result for $(a,b,c)=(-,-,-)$, the proof in the other cases 
being identical.
We first assume that $\sigma$ is also $C^1$ in $\theta$ (which can be 
achieved by convolution).
In this case, one has for all $(x_1,x_2,\theta)\in V$,
\[
\sigma^1(\bar x_1-\tfrac{\dx}{2},x_2,\theta)=
\sigma^1(x_1,x_2,\theta) - \int_{\bar x_1-\frac{\dx}{2}}^{x_1}
\partial_1 \sigma^1 (s,x_2,\theta) ds
\]
so that
\[
\dx \sigma^1(\bar x_1-\tfrac{\dx}{2},x_2,\theta)=
\int_{\bar x_1-\frac{\dx}{2}}^{\bar x_1+\frac{\dx}{2}}\sigma^1(x_1,x_2,\theta)
 - \int_{\bar x_1-\frac{\dx}{2}}^{\bar x_1+\frac{\dx}{2}}
(\bar x_1+\tfrac{\dx}{2}-s) \partial_1 \sigma^1 (s,x_2,\theta) ds
\]
Averaging over $x_2,\theta$, we deduce that
\[
\tau^1_- = \bar\sigma^1 - \frac{1}{\dx^2\dt}\int_{V} (\bar x_1+\tfrac{\dx}{2}-x_1) \partial_1 \sigma^1 dx_1dx_2d\theta.
\]
In the same way,
\[
\tau^2_- = \bar\sigma^2 - \frac{1}{\dx^2\dt}\int_{V} (\bar x_2+\tfrac{\dx}{2}-x_2) \partial_2 \sigma^2 dx_1dx_2d\theta,
\]
\[
\tau^\theta_- = \bar\sigma^\theta - \frac{1}{\dx^2\dt}\int_{V} (\bar \theta+\tfrac{\dt}{2}-\theta) \partial_\theta \sigma^\theta dx_1dx_2d\theta.
\]
Using that $\Div\sigma=0$, the latter can be rewritten
\[
\tau^\theta_- = \bar\sigma^\theta +\frac{1}{\dx^2\dt}\int_{V} (\bar \theta+\tfrac{\dt}{2}-\theta) \Div_x\sigma^x dx_1dx_2d\theta.
\]
The estimate~\eqref{eq:RT0vsAve} follows. If $\sigma$ is not $C^1$
in $\theta$, as before we can smooth $\sigma$, then in the limit
we will obtain~\eqref{eq:RT0vsAve} for all $\bar\theta$ such that
$\mu(\{\bar\theta\pm\dt/2\})=0$.
\end{proof}
\begin{corollary}\label{cor:RT}
Let $\tau(x,\theta)$ be the Raviart-Thomas extension of the fluxes $\tau^\bullet_\pm$
in $V$: then it holds
\begin{equation}\label{eq:RTvsAve}
\int_V|\tau-\bar\sigma|dxd\theta \le \sqrt{\dx^2+\dt^2}\int_{V} |\partial_1\sigma^1|+|\partial_2\sigma^2|
\end{equation}
 (for the same values of $\bar\theta$).
\end{corollary}
This is proven in the same way, as inside $V$, $\tau^\bullet(x,\theta)$
is a convex combination of the two fluxes $\tau^\bullet_\pm$.
Moreover, by construction since $\Div\sigma=0$, it is easy to check
that one also has $\Div\tau=0$.
The following is also immediate:
\begin{corollary}\label{cor:RT2}
Let $\tau(x,\theta)$ be the Raviart-Thomas extension of the fluxes $\tau^\bullet_\pm$
in $V$ and $\bar\tau$ the value in the middle of the cell (in other words,
\[
\bar\tau = \begin{pmatrix} \frac{\tau^1_- + \tau^1_+}{2} \\
\frac{\tau^2_- + \tau^2_+}{2} \\
\frac{\tau^\theta_- + \tau^\theta_+}{2} 
\end{pmatrix}
\]
is given by the average of the fluxes through the facets of $V$). Then
\begin{equation}\label{eq:RTcentervsAve}
|V||\bar\tau-\bar\sigma| \le \sqrt{\dx^2+\dt^2}\int_{V} |\partial_1\sigma^1|+|\partial_2\sigma^2|.
\end{equation}
\end{corollary}
\subsection{Consistent discretization of the energy $F$}

We now are in a position to define
almost consistent approximations of $F$. We will build
a discrete approximation which enjoys a sort of
discrete-to-continuum  $\Gamma$-convergence
property to the limiting functional $F$.

Assume to simplify $\Om$ is a convex set\footnote{This is not really 
important, as one could approximate $\sigma$ by smooth function
only inside $\Om$ and then let the corresponding set
invade $\Om$ in the limit.}, and even a rectangle $[0,a]\times [0,b]$,
$a,b>0$, which further simplifies our notation.

Let $u\in BV(\Om)$ and $\sigma_u$ be admissible for
$u$, such that $F(u)=\int_{\OS} h(\theta,\sigma_u)<\infty$
and first, for $\e>0$ fixed,  $\sigma_\e$ by convolution as in the proof
of Proposition~\ref{thm:smoothapprox}. In particular,
\[
\int_{\OS} |\partial_1\sigma^1_\e|+|\partial_2\sigma^2_\e|
\le \frac{c}{\e}\int_{\OS} |\sigma^1_u|+|\sigma^2_u|
\]
where $c=2\pi\int_{B_1}|\nabla\rho|dx$
depends only on the convolution kernel $\rho$.
Fix $\dx,\dt$ small enough, assume $\dt=2\pi/N_\theta$ for some
integer $N_\theta$, and consider all the volumes
$V_{i,j,k}$, defined in~\eqref{eq:defVijk} (with $S_{i,j}$ defined by~\eqref{defSij}),
%ICI 
%$V_{i,j,k}=[(i-\ha)\dx,(i+\ha)\dx)\times [(j-\ha)\dx,(j+\ha)\dx)\times [(k-\ha)\dt,(k+\ha)\dt)$
and which are inside $\OS$, for $(i,j,k)\in\J$~\eqref{eq:defJ}.
%i=1,\dots,N_1-1$, $j=1,\dots,N_2-1$, $k=0,\dots N_\theta-1$.
We define a Raviart-Thomas vector field from the (averaged) fluxes
 of $\sigma_\e$ through the facets of the volumes:
$\sigma^1_{i-\ha,j,k}$ through the facets $\F^1_{i-\ha,j,k}$,
%{(i-\ha)\dx\}\times [(j-\ha)\dx,(j+\ha)\dx)\times [(k-\ha)\dt,(k+\ha)\dt)$, $i=1,\dots,N_1$,
$\sigma^2_{i,j-\ha,k}$ through the facets $\F^2_{i,j-\ha,k}$,
%$[(i-\ha)\dx,(i+\ha)\dx)\times\{(j-\ha)\dx\}\times  [(k-\ha)\dt,(k+\ha)\dt)$, $j=1,\dots,N_2$,
and 
$\sigma^\theta_{i,j,k-\ha}$ through $\F^\theta_{i,j,k-\ha}$,
%$[(i-\ha)\dx,(i+\ha)\dx)\times [(j-\ha)\dx,(j+\ha)\dx)\times \{(k-\ha)\dt\}$
see~\eqref{eq:defF1}, \eqref{eq:defF2}, \eqref{eq:defFt}.
The latter flux is well-defined up to an infinitesimal vertical translation
of the origin of the discretization in $\theta$ (without loss
of generality we thus assume it is well defined).
The Raviart-Thomas field inside the cube is defined then
as in~\eqref{eq:RaTho}.
We also define $\bar\sigma_{i,j,k}=(\dx^{-2}\dt^{-1})\int_{V_{i,j,k}}\sigma_\e$
as the average of $\sigma_\e$ in $V_{i,j,k}$,
and let $\hat\sigma_{i,j,k}$ be defined by~\eqref{defhatsigma},
which corresponds to averaging the fluxes of the facets, or equivalently
to consider the value of the Raviart-Thomas extension in the middle
of the volume $V_{i,j,k}$.
% \begin{equation}\label{defhatsigma}
% \hat\sigma_{i,j,k} = \ha(\sigma^1_{i+\ha,j,k}+\sigma^1_{i-\ha,j,k},
% \sigma^2_{i,j+\ha,k}+\sigma^2_{i,j-\ha,k},\sigma^\theta_{i,j,k+\ha}+\sigma^\theta_{i,j,k-\ha}).
% \end{equation}

From Corollary~\ref{cor:RT2}, letting
\[
e_{i,j,k} = \hat\sigma_{i,j,k}-\bar \sigma_{i,j,k},
\]
one has
\[
\dx^2\dt\sum_{i,j,k}  |e_{i,j,k}|
\le c\frac{\sqrt{\dx^2+\dt^2}}{\e}\int_{\OS}|\sigma^x_u|.
\]
Moreover by Lemma~\ref{lem:controlmean}, one has
(we denote, for every $k$, $\theta_{k}=k \dt$
and $\theta_{k+\ha}=(k+\ha) \dt$)
\[
\dx^2\dt\sum_{i,j,k} \bar h(\bar\sigma_{i,j,k}\cdot\utheta_{k},
\cos\tfrac{\dt}{2}\bar\sigma^\theta_{i,j,k})
\le 
(1+C\dt^2) \int_{\OS} h(\theta,\sigma_u)
\]
(using that $\int_{\OS}h(\theta,\sigma_\e)\le\int_{\OS} h(\theta,\sigma_u)$,
\textit{cf} the proof of Prop.~\ref{thm:smoothapprox}).

Eventually, letting  now
\[
e'_{i,j,k} = e_{i,j,k} + (1-\cos\tfrac{\dt}{2})\bar\sigma^{\theta}_{i,j,k},
\]
which is such that
\begin{multline}\label{eq:errorcontrol}
\dx^2\dt\sum_{i,j,k} |e'_{i,j,k}|
\le
c\frac{\sqrt{\dx^2+\dt^2}}{\e}\int_{\OS}|\sigma^x_u| + \dt^2 \int_{\OS}|\sigma^\theta_u|
\\\le c\frac{\sqrt{\dx^2+\dt^2}}{\e}\int_{\OS}|\sigma_u|
\le c\frac{\sqrt{\dx^2+\dt^2}}{\gamma\e} F(u).
\end{multline}

We deduce that we can find an center-averaged Raviart-Thomas field 
$\hat\sigma$ and an error term $e'$
such that~\eqref{eq:errorcontrol} holds and
% \[
% \int_{\Om^\dx\times\Sp^1} \bar h((\sigma_{RT}-e')\cdot\utheta^{\dt},(\sigma_{RT}-e')^\theta)dxd\theta \le (1+C\dt^2) F(u)
% \]
%\begin{multline}
\[
\sum_{i,j,k} \bar h((\hat\sigma_{i,j,k}-e'_{i,j,k})\cdot\utheta_{k},(\hat\sigma_{i,j,k}-e'_{i,j,k})^\theta)
\le (1+C\dt^2) F(u)
\]
%\end{multline}
where $\theta^{\dt}=\sum_k \theta_k\chi_{\{k\dt,(k+1)\dt\}}$.
In particular if we introduce, for $\delta=(\delta_x,\delta_t)$ small,
the inf-convolution
\begin{equation}\label{eq:infconv}
\bar h_\delta(s,t) =  \min_{s',t'} \bar h(s-s',t-t') + \frac{\gamma}{(\dx^2+\dt^2)^{1/4}}\sqrt{s'^2+t'^2}
\end{equation}
we find that
\[
\dx^2\dt\sum_{i,j,k} \bar h_\delta(\hat\sigma_{i,j,k}\cdot\utheta_{k},\hat\sigma_{i,j,k}^\theta)
 \le (1+C\dt^2+ \tfrac{c}{\e}(\dx^2+\dt^2)^{1/4}) F(u).
\]
Moreover, one easily sees that
\begin{equation}\label{eq:sigmacone}
\hat\sigma_{i,j,k}^x \in \R_+\utheta_{k-\ha}+\R_+\utheta_{k+\ha}.
\end{equation}

\begin{remark}\textup{If $\bar h$ is $L$-Lipschitz (as it is the case
when $f$ as growth one, for instance if $f(t)=\gamma\sqrt{1+t^2}$),
then the inf-convolution step is not necessary. One directly obtains
\[
\dx^2\dt\sum_{i,j,k} \bar h(\hat\sigma_{i,j,k}\cdot\utheta_{k},\hat\sigma_{i,j,k}^\theta)
 \le \Big(1+C\dt^2+ \tfrac{cL}{\gamma\e}\sqrt{\dx^2+\dt^2}\Big)F(u).
\]
}
\end{remark}

Now, we check the consistency between $\hat\sigma$ and $u$..
By construction,  
$\dx\dt\sum_k \sigma^1_{i+\ha,j,k}$ is the flux of $Du^\perp$
through the edge $\{i+\ha\dx\}\times [(j-\ha)\dx,(j+\ha)\dx]$ in $\Omega$, hence
it is equal to the value $u((i+\ha)\dx,(j+\ha)\dx)-u((i+\ha)\dx,(j-\ha)\dx)$.
Accordingly, if we let, for
all $i,j$, $u_{i+\ha,j+\ha}^{\delta}:=u((i+\ha)\dx,(j+\ha)\dx)$,
we obtain that~\eqref{eq:compatsigmau} holds (with $u$ replaced
with $u^\delta$).
% \begin{equation}\label{eq:compatsigmau}
% \begin{cases}
% \dt\sum_k \sigma^1_{i+\ha,j,k} =
%  \frac{1}{\dx} \left(u^\delta_{i+\ha,j+\ha}-u^\delta_{i+\ha,j-\ha}\right)
% \\[2mm]
% \dt\sum_k \sigma^2_{i,j+\ha,k} = 
% -\frac{1}{\dx} \left(u^\delta_{i+\ha,j+\ha}-u^\delta_{i-\ha,j+\ha}\right).
% \end{cases}
% \end{equation}

Eventually we observe that the free divergence condition simply translates
as~\eqref{eq:freediv}
% \begin{equation}\label{eq:freediv}
% \frac{\sigma^1_{i+\ha,j,k}-\sigma^1_{i-\ha,j,k}}{\dx}+
% \frac{\sigma^2_{i,j+\ha,k}-\sigma^2_{i,j-\ha,k}}{\dx}+
% \frac{\sigma^\theta_{i,j,k+\ha}-\sigma^\theta_{i,j,k-\ha}}{\dt}=0
% \end{equation}
for all admissible $i,j,k$, as this is the global flux of
$\sigma_\e$ across the boundaries of $V_{i,j,k}$.

It is now easy to deduce the following upper approximation result:
\begin{proposition}
Let $u\in BV(\Om)$, $\sigma$ be admissible for $u$ and such that
\[ F(u)=\int_{\OS} h(\theta,\sigma)<\infty.\]
 Then  for $\delta=(\dx,\dt)\to 0$ one can find a discrete field 
$(\sigma^1_{i+\ha,j,k},\sigma^2_{i,j+\ha,k},\sigma^{\theta}_{i,j,k+\ha})$ and
a discrete image $u^\delta_{i+\ha,j+\ha}$ with 
\begin{equation}\label{eq:limitu}
\sum_{i,j} u^\delta_{i+\ha,j+\ha}\chi_{[i\dx,(i+1)\dx)\times[j\dx,(j+1)\dx)}\to u
\end{equation}
(strongly in $L^2(\Om)$)
and such that for all $i,j$, \eqref{eq:compatsigmau} holds, for all
$i,j,k$, \eqref{eq:sigmacone} and \eqref{eq:freediv} hold, and:
\begin{equation}
\limsup_{\delta\to 0}
\dx^2\dt\sum_{i,j,k} \bar h_\delta(\hat\sigma_{i,j,k}\cdot\utheta_{k},\hat\sigma_{i,j,k}^\theta)
\le F(u)
\end{equation}
where $\hat\sigma$ is defined by~\eqref{defhatsigma}, and
where $\bar h_\delta$ is defined in~\eqref{eq:infconv} (or is $\bar h$ in case
it is Lipschitz).
\end{proposition}

To show that the discretization is consistent, we must now show a similar
lower bound: namely that given any $u$ and $u^\delta,\sigma,\hat\sigma$
which satisfy~\eqref{defhatsigma},
\eqref{eq:sigmacone}, \eqref{eq:compatsigmau},
and~\eqref{eq:limitu} (weakly, for instance as distributions), then one has
\begin{equation}\label{eq:gliminf}
\liminf_{\delta\to 0}
\dx^2\dt\sum_{i,j,k} \bar h_\delta(\hat\sigma_{i,j,k}\cdot\utheta_{k},\hat\sigma_{i,j,k}^\theta)
\ge F(u).
\end{equation}

%%% THIS LATER:
% We will show a bit more, which is that if
% \begin{equation}\label{eq:upperbound}
% \sup_{\delta} 
% \dx^2\dt\sum_{i,j,k} \bar
% h(\hat\sigma_{i,j,k}\cdot\utheta_{k},\hat\sigma_{i,j,k}^\theta)
% <+\infty
% \end{equation}
% then one can build a $u^\delta$ with~\eqref{eq:compatsigmau},
% the convergence in~\eqref{eq:limitu} and~\eqref{eq:gliminf}.
% Also, we will prove~\eqref{eq:gliminf} with $\bar h_\delta$ replaced
% with $\bar h\ge \bar h_\delta$, for simplicity. This is seemingly
% weaker, however it is not hard to show that if it holds and the bound~\eqref{eq:upperbound}
% is true, then also~\eqref{eq:gliminf} holds.

A first obvious remark is that the field
\[
\hat\sigma^\delta := \sum_{i,j,k} \hat\sigma_{i,j,k}
%\chi_{[(i-\ha)\dx,(i+\ha)\dx)\times[(j-\ha)\dx,(j+\ha)\dx)\times[(k-\ha)\dt,(k+\ha)\dt)}
\chi_{V_{i,j,k}}
\]
is bounded in measure, and hence, up to subsequences, converges
(weakly-$*$) to
a measure $\sigma$. It is then easy to deduce from~\eqref{eq:sigmacone}
and the convexity of $\bar h$ that
\[
\int_{\OS} h(\theta,u)\le\liminf_{\delta\to 0}
\liminf_{\delta\to 0}
\dx^2\dt\sum_{i,j,k} \bar h_\delta(\hat\sigma_{i,j,k}\cdot\utheta_{k+\ha},\hat\sigma_{i,j,k}^\theta).
\]
One  can also check that $\Div \sigma=0$ by passing to the
limit in~\eqref{eq:freediv} (after a suitable integration against
a smooth test function, exactly as in~\eqref{eq:consistentderiv} below).
Hence it is enough to show that the limiting $\sigma$ is compatible with $u$.

But this is quite obvious from~\eqref{eq:compatsigmau}, which one can
integrate against a smooth test function, then ``integrate by part''
before passing to the limit. More precisely, for $\varphi\in C^\infty_c(\Om)$,
one has (dropping the superscripts $\delta$
and denoting $\varphi_{i,j}=(1/\dx^2)\int_{S_{i,j}} \varphi(x) dx$):
\begin{multline}
\label{eq:consistentderiv}
\begin{aligned}
\hspace{-2mm}\int_{\OS}\varphi(x) \hat\sigma^1 dx d\theta
&=\sum_{i,j,k} \hat\sigma^1_{i,j,k}\int_{V_{i,j,k}}\varphi(x) dx d\theta 
\\ &=\dx^2\dt \sum_{i,j} \varphi_{i,j} \sum_k \dt \frac{\sigma^1_{i+\ha,j,k}+\sigma^1_{i-\ha,j,k}}{2}
\\ = &\, \frac{\dx^2\dt}{2} \sum_{i,j} \varphi_{i,j} \left(\frac{u_{i+\ha,j+\ha}-u_{i+\ha,j-\ha}}{\dx}
+\frac{u_{i-\ha,j+\ha}-u_{i-\ha,j-\ha}}{\dx}\right)
\\ & =-\dx^2\dt \sum_{i,j} \frac{u_{i+\ha,j+\ha}+u_{i-\ha,j+\ha}}{2} \frac{\varphi_{i,j+1}-\varphi_{i,j}}{\dx}
\end{aligned}
\\\to -\int_\Om u(x)\partial_2\varphi(x) dx
\end{multline}
as $\delta\to 0$.

Eventually, we need to show
a compactness property, which is that if
 \begin{equation}\label{eq:upperbound}
 \sup_{\delta} 
 \dx^2\dt\sum_{i,j,k} \bar
 h(\hat\sigma_{i,j,k}\cdot\utheta_{k},\hat\sigma_{i,j,k}^\theta)
 <+\infty
 \end{equation}
the discrete image $u^\delta$ which is recovered from~\eqref{eq:compatsigmau}
(up to a constant) converges to a $u(x)$, $x\in\Om$
(in a weak sense which will be made clear). The point here is that
a priori, from~\eqref{eq:upperbound} and~\eqref{eq:controlbelow}, one
has only
\begin{multline}\label{eq:oscillatingTV}
\sup_{\delta} \dx^2 \sum_{i,j}
\Bigg( \left(\frac{u^\delta_{i+\ha,j+\ha}-u^\delta_{i+\ha,j-\ha}}{\dx}
+\frac{u^\delta_{i-\ha,j+\ha}-u^\delta_{i-\ha,j-\ha}}{\dx}\right)^2
\\+ \left(\frac{u^\delta_{i+\ha,j-\ha}-u^\delta_{i-\ha,j-\ha}}{\dx}
+\frac{u^\delta_{i+\ha,j+\ha}-u^\delta_{i-\ha,j+\ha}}{\dx}\right)^2
\Bigg)^{\ha} <+\infty
\end{multline}
so that the discrete total variation of $u^\delta$ is
a priori not well controlled. However, one can easily check
that the kernel of the operator which appears in the energy~\eqref{eq:oscillatingTV} is two-dimensional,
and made of the oscillating discrete images
\begin{equation}\label{eq:kernel}
v_{i-\ha,j-\ha}^\delta = \alpha + \beta (-1)^{i+j},
\end{equation}
$\alpha,\beta\in \R^2$. Hence it is possible to show that
one can decompose
$u^\delta$ as a sum of a non-oscillating function with
zero average $\bar u^\delta$ and an oscillation $v^\delta$,
and obtain a strong control on the discrete total variation of
$\bar u^\delta$. Therefore one easily deduce that any suitably built
continuous extension of $\bar u^\delta$ will converge to some
$u$ strongly in $L^p(\Om)$, for any $p<2$ (as $BV(\Om)$ is compactly
embedded in such spaces), and weakly in $L^2(\Om)$.

In addition, any control on the average of $u^\delta$ and on its oscillation
(which cannot
be given by~\eqref{eq:oscillatingTV} and has to come from other
terms in the energy, such as a boundary condition or a penalization:
note that it is enough to control two adjacent pixels)
will ensure in addition that $v^\delta$ remains bounded and
converges (only weakly in $L^p(\Om)$,
 if the control is only on the $L^p$ norm, however in this case it
is obvious that the oscillating term in~\eqref{eq:kernel} goes to zero
and $v^\delta$ can only go to a constant).

To sum up, we have shown the following.
\begin{proposition}
For $\delta \to 0$, assume we are given $\sigma^\delta$, $u^\delta$
and $\hat\sigma^\delta$ with~\eqref{defhatsigma}, \eqref{eq:compatsigmau}, which
in addition satisfy~\eqref{eq:sigmacone} and \eqref{eq:freediv},
and~\eqref{eq:upperbound}. Then,
up to an oscillating function $v^\delta$ of the form~\eqref{eq:kernel},
there is $u\in BV(\Om)$ such that
$u^\delta\to u$, and~\eqref{eq:gliminf} holds.
\end{proposition}

\begin{remark}\textup{In practice, we did not use the inf-convolutions
$\bar h_\delta$ (only $\bar h$) in our discrete
scheme.
Also, we replaced the constraint~\eqref{eq:sigmacone}
with the stronger constraint $\hat{\sigma}^x_{i,j,k}\in \R_+\utheta_k$,
after having experimentally
observed that there was no qualitative difference in the output.
It seems the results we compute are still consistent with what is expected
from the energy.}
\end{remark}

\section{Smirnov's theorem in $\OS$}\label{app:Smirnov}

In this whole paper $\Om\subset\R^2$ is assumed to be a Lipschitz set.
In particular,
locally its boundary can be represented as the subgraph 
$\{(x,y): y<h(x)\}$ of a Lipschitz function $h$. Consider
a ball $B$ where this representation holds
and assume first $h$ is $C^1$, then one can extend
in $B$ a bounded Radon measure $\sigma$ with $\Div \sigma=0$ into
$\tilde \sigma$ defined (for $\psi\in C_c^0(B;\R^2)$)
\[
\int_B \tilde\sigma\cdot \psi := 
\int_{B\cap \Om}\sigma\cdot \left(\psi(x,y)-
\begin{pmatrix}1 & 2h'(x)\\ 0 & -1\end{pmatrix}\psi(x,2h(x)-y)\right).
\]
Then, it is standard that $\Div\tilde\sigma=0$ in $B$, indeed, if
$\vp\in C_c^1(B)$, one has that
\[
\int_B \tilde\sigma\cdot \nabla\vp = 
\int_{B\cap\Om} \sigma\cdot\nabla\left[\vp(x,y)-\vp(x,2h(x)-y)\right].
\]
The function $\vp^s(x,y):=\vp(x,y)-\vp(x,2h(x)-y)$ is $C^1$ and vanishes on $\partial\Om$,
hence this expression is zero: Indeed if for $\tau>0$ one lets
$\vp^s_\tau(x,y)=S_\tau(\vp^s(x,y))$ where $S_\tau\in C^\infty(\R)$
is a smooth approximation of a ``shrinkage operator'':
\[
S_\tau(t) =\begin{cases} t-\tau 
& \textup{ if } t\ge \tfrac{3}{2}\tau\,,
\\ 0 &   \textup{ if } |t|< \tfrac{1}{2}\tau\,,
\\ t+\tau & \textup{ if } t\le - \tfrac{3}{2}\tau\,,
\end{cases}
\]
with smooth and $1$-Lipschitz interpolation in $\pm[\tau/2,3\tau/2]$,
then $\vp^s_\tau\in C_c^1(B\cap\Om)$ so that
\[
\int_{B\cap \Om} \sigma \cdot\nabla \vp^s_\tau=0
\]
and, using $\nabla \vp^s_\tau = S'_\tau(\vp^s)\nabla \vp^s$
\[
\int_{B\cap\Om}\sigma\cdot (\nabla \vp^s-\nabla\vp^s_\tau)
\le C |\sigma|(B\cap\Om\cap\{0<|\vp^s|< 3\tau/2\})\to 0
\]
as $\tau\to 0$, showing our claim. Hence $\Div\tilde\sigma=0$.
If $h$ is not $C^1$ but just Lipschitz, one can approximate it
from below by smooth functions $h_n$, build in such a way a sequence
$\sigma_n$ of extensions of $\sigma\raisebox{-2pt}{$|$} {}_{\{y<h_n(x)\}}$
and pass to the limit to deduce that the extension still exists.

Using cut-off functions, one can therefore assume that $\sigma$ can be extended
into a field $\tilde\sigma$ which is a measure in $\R^2$ with
free divergence in a neighborhood of $\Om$.

A similar construction
would allow to extend a field $\sigma\in\M(\Om\times \R^2;\R^4)$ 
to $\M(\R^4;\R^4)$ with free divergence (either in a neighborhood
or $\Om\times \R^4$, or even everywhere).
This remark allows to localize Smirnov's theorems in~\cite{Smirnov}.

Consider indeed now a free divergence field $\sigma\in \M(\OS;\R^3)$.
It can be seen, after extension, as a field in $\M(\R^2\times\R^2;\R^4)$ with
$\spt\sigma\subseteq \R^2\times \Sp^1$.

As in Smirnov's paper~\cite{Smirnov}, for $l>0$
we introduce $\Cl$ the set of oriented curves $\gamma$
in $\R^4$ with length $l$, with the topology corresponding
to the weak convergence of the measures $\tau_\gamma\H^1\restr\gamma$.
Then, thanks to~\cite[Theorem A]{Smirnov},
$\sigma$ can be decomposed as
\[
\sigma = \int_{\Cl} \lambda d\mu(\lambda), \quad
|\sigma| = \int_{\Cl} |\lambda| d\mu(\lambda), 
\]
for some measure $\mu$ on $\Cl$. Moreover thanks to Remark~5 in~\cite{Smirnov},
$\mu$-a.e.~curve in the decomposition lies in $\R^2\times\Sp^1$.

%extremal \textit{elementary solenoids}~\cite{Smirnov} or finite
%length curves connecting two endpoints. 
%and (cf~\cite[Remark 5]{Smirnov}) shows that $\mu$-a.e.~all charges $\lambda$
%which appear in the decompositions

For this work, it is enough to consider $l=1$. Moreover,
if we restrict then all these measures to $\OS$ (and take for $\mu$
the corresponding marginal), we get a decomposition on curves
of length less or equal to 1 (possibly entering/exiting the domain).
By a slight abuse of notation we still denote $\Cone$ such a
set of curves.
One finds that
\begin{equation}\label{eq:SmirnovCurves}
\sigma = \int_{\Cone} \lambda d\mu(\lambda), \quad
|\sigma| = \int_{\Cone} |\lambda| d\mu(\lambda), 
\end{equation}
with now $\lambda\in\Cone$, the curves of length at most one in $\OS$.
\smallskip

\begin{remark}\textup{
Theorem~B in~\cite{Smirnov} is a more precise statement.
It shows that one can obtain a similar decomposition with
now curves $\lambda$ with $\Div\lambda=0$ a.e.:
being either finite curves entering and exiting the domain,
or ``elementary solenoids'', which are objects of the form
\begin{align*}
& \lambda= \M-\lim_{k\to\infty}\frac{1}{2k} f_\sharp \overrightarrow{[-k,k]} \restr \OS\\
& \textup{Lip}(f)\le 1 \\
& \textup{var}(\lambda)=1\, \\
& f(\R)\subset \spt(\lambda),
\end{align*}
(in particular one should have $|f'(t)|=1$ \ale),
meaning that for any $\vp\in C_c^1(\OS)$,
\begin{equation}
\label{eq:ES}
\lambda(\vp) = \lim_{s\to\infty} \frac{1}{2s}\int_{-s}^s \sca{f'(t)}{\vp(f(t))}dt.
\end{equation}
This expresses that either $\lambda$ is defined by the closed curve
$f(\R)$ (if $f$ is periodic), or $\lambda$ is a limit of curves which
densify and do not loose mass in the limit. We do not need such a
precise result for our construction.
}
\end{remark}

\end{document}